\documentclass[11pt]{article}
\usepackage{artmods}
\usepackage{array,latexsym,amsmath,amssymb}
\usepackage{amsfonts,cite}
\usepackage{graphicx,longtable,pdflscape}
\usepackage{mathtools}
\usepackage{algorithm}
\usepackage{algpseudocode}
\usepackage{enumerate}
\usepackage{float}
\date{UCSD Center for Computational Mathematics\\
      Technical Report CCoM-21-01\\ October 1, 2021}

\title{Projected-Search Methods for \\ Bound-Constrained Optimization}
\author{Michael W. Ferry\thanks{NVIDIA Corporation, Hillsboro, Oregon
    (\texttt{michael@mwferry.com}).}
    \and
    Philip E. Gill%
    \thanks{Department of Mathematics, University of
    California, San Diego, La Jolla, CA 92093-0112 (\texttt{pgill@ucsd.edu},
    \texttt{elwong@ucsd.edu}, \texttt{miz151@ucsd.edu})
    Research supported in part by National Science Foundation grants
    DMS-0915220 and DMS-1318480. The content is solely
    the responsibility of the authors and does not necessarily represent
    the official views of the funding agencies. }
    \and
    Elizabeth Wong%
    \addtocounter{footnote}{-1}\footnotemark
    \and
    Minxin Zhang%
    \addtocounter{footnote}{-1}\footnotemark}

\newtheorem{proposition}{Proposition}[section]

\begin{document}

\maketitle

\setcounter{page}{1}

\newcommand{\blackslug}{\hbox{\hskip 1pt \vrule width 4pt height 6pt depth 1.5pt
  \hskip 1pt}}
\newcommand{\pmat}[1]{\begin{pmatrix}#1\end{pmatrix}}

\makeatletter
\newcommand{\minim}{\mathop{\operator@font{minimize}}}
\newcommand{\subject}{\mathop{\operator@font{subject\ to}}}

\providecommand{\diag} {\mathop{\operator@font{diag}}}
\newcommand{\supmax}{^{\operator@font{max}}}
\newcommand{\supmin}{^{\operator@font{min}}}
\newcommand{\submax}{_{\operator@font{max}}}
\newcommand{\submin}{_{\operator@font{min}}}
\newcommand{\subminus}{_{\scriptscriptstyle -}}
\newcommand{\subplus }{_{\scriptscriptstyle +}}
\newcommand{\supminus }{^{\scriptscriptstyle -}}
\newcommand{\supplus }{^{\scriptscriptstyle +}}
\newcommand{\subzero }{_{\scriptscriptstyle 0}}
\makeatother
\newcommand{\superplus}{^{+}}
\newcommand{\superminus}{^{-}}

\newcommand{\st}{:}
\newcommand{\sub}[1]{_{\mbox{\protect\tiny #1}}}
\newcommand{\minimize}[1]{{\displaystyle\minim_{#1}}}

\newcommand{\LBFGSB}{\texttt{L-BFGS-B}}
\providecommand{\Matlab}{{\sc Matlab}}
\newcommand{\etal}{et al.}  %%% No italics!!  Also, must say \etal\

\newcommand{\D}{_{\scriptscriptstyle D}}
\newcommand{\M}{_{\scriptscriptstyle M}}
\renewcommand{\L}{_{\scriptscriptstyle L}}

\newcommand{\drop}{^{\null}}
\newcommand{\Set}[1]{\big\{\, #1 \,\big\}}
\newcommand{\starsymbol}{\ast}
\newcommand{\superstar}{^\starsymbol}

\providecommand{\varPi}{{\mathit\Pi}}
\providecommand{\varDelta}  {{\mathit\Delta}}
\newcommand{\Deltait}{\varDelta}
\newcommand{\Omegait}{\varOmega}
\newcommand{\Piit}{\varPi}

\newcommand{\sgap}{\;}
\newcommand{\mgap}{\;\;}
\newcommand{\bgap}{\;\;\;}

\newcommand{\abs}[1]{|#1|}
\newcommand{\Mod}[1]{\left|#1\right|}
\newcommand{\norm}[1]{\|#1\|}
\newcommand{\dualnorm}[1]{\norm{#1}\D}
\newcommand{\infnorm}[1]{\norm{#1}_{\infty}\drop}
\newcommand{\elem}[1]{[\,#1\,]} % \elem{x\B}_j element j of vector x\B
\newcommand{\elemd}[1]{[#1]^{\null}} %  \elem{x\B}_j element j of vector x\B

\newcommand{\alphabar}{\skew3\bar\alpha}
\newcommand{\alphahat}{\skew3\hat\alpha}
\newcommand{\alphastar}{\alpha\superstar}
\newcommand{\xstar}{x\superstar}

\newcommand{\Grad}{\nabla\!}
\newcommand{\Hess}{\nabla^2\!}
\newcommand{\half}  {{\textstyle\frac12}}

\newcommand{\id}{_i\drop}
\newcommand{\jm}[1]{_{j-#1}}
\newcommand{\kd}{_k\drop}
\newcommand{\kdp}[1]{_{k+#1}\drop}
\newcommand{\kdm}[1]{_{k-#1}\drop}
\newcommand{\km}[1]{_{k-#1}}
\newcommand{\kp}[1]{_{k+#1}}

\newcommand{\subM}{_{\scriptscriptstyle M}}

\newcommand{\dottimes}{\mathop{\mbox{{\,\boldmath$\cdot$}\,}}}
\newcommand{\dotdiv}{\mathop{\mbox{{\,\boldmath$\cdot/$}\,}}}
\newcommand{\inv}{^{-1}}
\newcommand{\T}{^T\!}
\newcommand{\Tinv}{^{-T}\!}
\newcommand{\subU}{_{\scriptscriptstyle U}}
\newcommand{\U}{\subU}
\newcommand{\up}[1]{^{(#1)}}

\newcommand{\Hhat}{\skew6\widehat H}

\newcommand{\Gbar}{\bar G}
\newcommand{\setA}{\mathcal{A}}
\newcommand{\setI}{\mathcal{I}}
\newcommand{\setG}{\mathcal{G}}
\newcommand{\setGbar}{{\mathcal\Gbar}}
\newcommand{\setK}{\mathcal{K}}
\newcommand{\setL}{\mathcal{L}}
\newcommand{\setP}{\mathcal{P}}
\newcommand{\setW}{\mathcal{W}}

\newcommand{\halfthinsp}{\mskip 0.5 mu}  %  half     thin space
\newcommand{\lwords}[1]{\text{#1}\mgap}     % text automatically adjusts the size
\newcommand{\rwords}[1]{\mgap\text{#1}}
\newcommand{\lwordss}[1]{\text{#1}\quad}    % text automatically adjusts the size
\newcommand{\rwordss}[1]{\quad\text{#1}}
\newcommand{\lwordsss}[1]{\text{#1}\qquad}  % text automatically adjusts the size
\newcommand{\rwordsss}[1]{\qquad\text{#1}}
\newcommand{\words}[1]{\mgap\text{#1}\mgap}
\newcommand{\wordss}[1]{\quad\text{#1}\quad}
\newcommand{\wordsss}[1]{\qquad\text{#1}\qquad}

\newcommand{\algstrut}{\rule[-.6ex]{0pt}{2.5ex}}% Total height 3.1
\newcommand{\Astrutu}{\rule{0pt}{2.75ex}}
\newcommand{\agap}{\hspace{1.5ex}}

\newcommand{\REPEAT}{\textbf{repeat}\Astrut}
\newcommand{\WHILE}{\textbf{while}\hskip2pt}
\newcommand{\FOR}{\textbf{for}\hskip2pt}
\newcommand{\IF}{\textbf{if}\hskip2pt}

\newcommand{\logical}[1]{\{\mbox{#1}\}}
\newcommand{\proc}[1]{\setbox4=\hbox{\noindent\strut#1}\boxit{\box4}}
\newcommand{\procc}[1]{\setbox4=\vbox{\hsize 20pc \noindent\strut#1}\boxit{\box4}}
\newcommand{\DO}{\hskip2pt\textbf{do}\hskip2pt\Astrut}
\newcommand{\UNTIL}{\textbf{until}\hskip2pt}
\newcommand{\BEGIN}{\textbf{begin}}
\newcommand{\END}{\textbf{end}}
\newcommand{\ENDIF}{\textbf{end if}}
\newcommand{\ENDWHILE}{\textbf{end while}}
\newcommand{\ENDREPEAT}{\textbf{end repeat}}
\newcommand{\ENDDO}{\textbf{end do}}
\newcommand{\ENDFOR}{\textbf{end (for)}}
\newcommand{\THEN}{\hskip2pt\textbf{then}\hskip2pt}
\newcommand{\ELSE}{\hskip2pt\textbf{else}\hskip2pt}
\newcommand{\FIRSTELSE}{\textbf{else}\hskip2pt}
\newcommand{\ELSEIF}{\textbf{else if}}
\newcommand{\RETURN}{\textsf{return}}
\newcommand{\FIRST}{\hskip-2pt}  % \FIRST\ELSE if \ELSE starts a line

\newcommand{\STOP}{\mathsf{stop}}
\newcommand{\BREAK}{\mathsf{break}}
\newcommand{\TRUE}{\mathsf{true}}
\newcommand{\FALSE}{\mathsf{false}}
\newcommand{\true}{\mathsf{true}}
\newcommand{\false}{\mathsf{false}}
\newcommand{\NOT}{\mathop{\mathsf{not}\,}}
\newcommand{\AND}{\mathop{\;\mathsf{and}\;}}
\newcommand{\OR}{\mathop{\;\mathsf{or}\;}}

\input{quasiwolfemacros}

\vspace{-20pt}

\begin{abstract}
\noindent
Projected-search methods for bound-constrained optimization are based on
performing a search along a piecewise-linear continuous path obtained by
projecting a search direction onto the feasible region. A potential benefit
of a projected-search method is that many changes to the active set can be
made at the cost of computing a single search direction.

As the objective function is not differentiable along the search path, it
is not possible to use a projected-search method with a step that satisfies
the Wolfe conditions, which require the directional derivative of the
objective function at a point on the path.  For this reason, methods based
in full or in part on a simple backtracking procedure must be used to give
a step that satisfies an ``Armijo-like'' sufficient decrease condition.  As
a consequence, conventional projected-search methods are unable to exploit
sophisticated safeguarded polynomial interpolation techniques that have
been shown to be effective for the unconstrained case.

This paper concerns the formulation and analysis of projected-search
methods based on a new \emph{quasi}-Wolfe line search that is appropriate
for piecewise differentiable functions.  The behavior of the line search is
similar to that of a conventional Wolfe line search, except that a step is
accepted under a wider range of conditions. These conditions take into
consideration steps at which the restriction of the objective function on
the search path is not differentiable. Two new classes of method are
proposed that may be broadly categorized as active-set methods and
interior-point methods.  Computational results are given for two specific
methods from these general classes: a projected-search active-set method
that uses a limited-memory quasi-Newton approximation of the Hessian; and a
projected-search primal-dual interior-point method.  The results show that
in these contexts, a quasi-Wolfe line search is substantially more
efficient and reliable than an Armijo line search.

\end{abstract}

\thispagestyle{plain}
\markboth{Projected-Search Methods}{}
\newpage

\section{Introduction}
%---------------------
\medskip
This paper describes a new framework for the development of a general class of
projected-search methods for the bound-constrained problem
\begin{gather}
     \minimize{x\in \Re^n}\mgap f(x)
\bgap\subject             \mgap  x\in \FR, \tag{BC}\label{BC}
\end{gather}
where $f: \Re^n \mapsto \Re$ is a twice-continuously differentiable
function and $\FR = \Set{ x\in\Re^n \st \bL \le x \le \bU}$ for vectors of
lower and upper bounds such that $\bL \le \bU$ (with all inequalities
defined componentwise). The gradient of $f$ at $x$ is denoted by $\Grad
f(x)$. The active set of variables on their bounds at $x\in \FR$ is denoted
by $\setA(x)$, i.e., $\setA(x) = \tSet{ i \st x_i = \bL_i \words{or} x_i =
  \bU_i }$.

Projected-search methods for problem (\ref{BC}) generate a sequence of
feasible iterates $\tSet{x_k}_{k=0}^\infty$ such that $x\kp1 = \ProjF(x_k +
\alpha_k p_k)$, where $p_k$ is a descent direction for $f$ at $x_k$,
$\alpha_k$ is a scalar step length, and $\ProjF(x)$ is the projection of
$x$ onto the feasible region, i.e.,
\[
 \elem{\ProjF(x)}_i
   =   \begin{cases}
         \bL_i & \lwords{if} x_i < \bL_i, \\
         \bU_i & \lwords{if} x_i > \bU_i, \\
           x_i & \lwords{otherwise.}
       \end{cases}
\]
The new iterate may be written as $x\kp1 = x_k(\alpha_k)$, where
$x_k(\alpha)$ denotes the vector $x_k(\alpha) = \ProjF(x_k + \alpha p_k)$.
A potential benefit of a projected-search method is that many changes to
the active set can be made at the cost of computing a single search
direction.  The projected-search methods of Goldstein~\cite{Gol64}, Levitin
and Polyak~\cite{LevP66}, and Bertsekas~\cite{Ber76} are based on using the
gradient-descent direction $p_k = -\Grad f(x_k)$.  Bertsekas~\cite{Ber82a}
proposes a method based on computing $p_k$ using a Newton-like method.
Calamai and Mor\'e~\cite{CalM87} consider methods that identify the optimal
active set using a projected-search method and then switch to Newton's
method.  Projected-search methods based on computing $p_k$ using a
quasi-Newton method are proposed by Ni and Yuan~\cite{NiY97}, Kim, Sra and
Dhillon~\cite{KimSD10}, Ferry~\cite{Fer11}, and Ferry~\etal~\cite{FerGWZ20b}.

Many methods for unconstrained minimization generate a sequence of iterates
$\tSet{x_k}_{k=0}^\infty$ such that $x\kp1$ is chosen to give a decrease in
$f$ that is at least as good as a fixed fraction $\cA$ $(0 < \cA < \half)$ of
the decrease in the local affine model $f(x_k) + \Grad f(x_k)\T (x - x_k)$.
If $x\kp1$ is computed as $x\kp1 = x_k + \alpha_k p_k$, where $p_k$ is a
descent direction for $f$ at $x_k$ and $\alpha_k$ is a positive scalar,
then the decrease condition may be written as
\begin{equation}              \label{eq:Armijo}
  f(x_k + \alpha_k p_k)
    \le f(x_k) + \alpha_k\cA \Grad f(x_k)\T p_k,
\end{equation}
which is known as the Armijo condition (see, e.g., Armijo~\cite{Arm66},
Ortega and Rheinboldt~\cite{OrtR00}).  Most Armijo line searches are
implemented as a simple backtracking procedure in which an initial step is
reduced by a constant factor until the Armijo condition (\ref{eq:Armijo})
is satisfied.  Alternatively, backtracking may be used in conjunction with
a simple quadratic interpolation scheme using $f(x_k)$, $\Grad f(x_k)\T
p_k$ and $f(x_k + \alpha p_k)$ at each trial $\alpha$ (see Dennis and
Schnabel~\cite{DenS96}).

Many practical methods use an $\alpha_k$ that satisfies an additional
condition on the directional derivative $\Grad f(x_k + \alpha_k p_k)\T
p_k$. In particular, the strong Wolfe conditions require that $\alpha_k$
satisfies both the Armijo condition (\ref{eq:Armijo}) and
\begin{equation} \label{eq:strongWolfe}
  \abs{\Grad f(x_k + \alpha_k p_k)\T p_k}
    \le \cW\abs{\Grad f(x_k)\T p_k},
\end{equation}
where $\cW$ is a preassigned scalar such that $\cW \in (\cA,1)$ (see, e.g.,
Wolfe~\cite{Wol69}, Mor{\'e} and Thuente~\cite{MorT94}, and Gill \etal{}
\cite{GilMSW82}).  The strong Wolfe conditions allow $\cW$ to be chosen to
vary the accuracy of the step.  If $\cA$ is fixed at a value close to zero
(e.g., $10^{-4}$), then a value of $\cW$ close to $\cA$ gives a ``tighter''
or more accurate step with respect to closeness to a critical point of
$\Grad f(x_k + \alpha p_k)\T p_k$.  A value of $\cW$ close to one results
in a ``looser'' or more approximate step.  A Wolfe line search is able to
exploit sophisticated safeguarded polynomial interpolation techniques to
provide methods that are more reliable and efficient than those based on
backtracking (see, e.g., Hager~\cite{Hag89} and Mor\`e and
Thuente~\cite{MorT94}).

\smallskip
In a projected-search method, the function $x_k(\alpha)$ defines a
piecewise-linear continuous path, and the function $f\big(x_k(\alpha)\big)$
is not necessarily differentiable along $x_k(\alpha)$.  In particular,
$f\big(x_k(\alpha)\big)$ may have a ``kink'' at any $\alpha > 0$ at which
$\elem{p_k}_i\ne 0$ and either $\elem{x_k+\alpha p_k}_i = \bL_i$ or
$\elem{x_k + \alpha p_k}_i = \bU_i$.  This implies that it is not possible
to use a line search based on the conventional Wolfe conditions.  Thus,
existing projected-search methods are restricted to using a search based on
satisfying an Armijo-like condition along the path $x_k(\alpha)$. For the
case where $p_k=-\Grad f(x_k)$, a commonly used Armijo-like condition is
\begin{equation}              \label{eq:Armijo-like}
  f\big(x_k(\alpha_k)\big) \le f(x_k) + \cA \Grad f(x_k)\T (x_k(\alpha)-x_k),
\end{equation}
proposed by Bertsekas~\cite{Ber76} (see also, Calamai and
Mor\'e~\cite{CalM87}).  However, for a general $p_k$, this may not be a
sufficient-decrease condition for a backtracking search as there is no
guarantee that the second term on the right-hand side of
(\ref{eq:Armijo-like}) is negative if the path $x_k(\alpha)$ changes
direction.  An Armijo-like condition that is appropriate for a general
descent direction $p_k$ is
\begin{equation}              \label{eq:quasiArmijo}
  f\big(x_k(\alpha_k)\big) \le f(x_k) + \alpha_k\cA \Grad f(x_k)\T p_k
\end{equation}
(see, e.g., Ni and Yuan~\cite{NiY97} and Kim, Sra and
Dhillon~\cite{KimSD10}).  Throughout the following discussion,
(\ref{eq:quasiArmijo}) is referred to as the \emph{quasi-Armijo} condition.
If $\gamma$ and $\sigma$ denote fixed parameters such that $\gamma > 0$ and
$\sigma\in (0,1)$, then a quasi-Armijo step has the form $\alpha_k =
\gamma\sigma^{t_k}$, where $t_k$ is the smallest nonnegative integer such
that the quasi-Armijo condition (\ref{eq:quasiArmijo}) is satisfied.  Other
sufficient decrease conditions have been proposed. For example,
Bertsekas~\cite{Ber82a} considers an Armijo-like condition based on a
combination of (\ref{eq:Armijo-like}) and (\ref{eq:quasiArmijo}), with the
term (\ref{eq:Armijo-like}) defined with components of a scaled
gradient-descent direction.

\subsection{Contributions and organization of the paper}
%-------------------------------------------------------
Several contributions are made to the design and analysis of algorithms for
constrained optimization.
(i) A \emph{quasi-Wolfe} search is proposed that is specifically designed
for use with a piecewise-linear continuous search path.  The behavior of
the search is similar to that of a conventional Wolfe line search, except
that a step is accepted under a wider range of conditions that take into
account steps at which $f$ is not differentiable.  As in the unconstrained
case, the quasi-Wolfe step can be computed using safeguarded polynomial
interpolation and the accuracy of the step can be adjusted.
(ii) Two new classes of projected-search method are proposed that may be
broadly categorized as active-set methods and interior-point methods.  The
class of active-set methods is characterized by the use of a descent
direction $d_k$ computed with respect to a perturbed or extended active set
(a similar set is used by Bertsekas~\cite{Ber82a}).  The vector $d_k$ may
be computed in many ways, e.g., using an exact or modified Newton-like
method or a quasi-Newton method.  This direction is used as the basis for
the computation of a search direction $p_k$, and an associated step length
$\alpha_k$ such that $f\big(\ProjF(x_k + \alpha_k p_k)\big) < f(x_k)$.  The
class of interior-point methods project the underlying Newton or
Newton-like direction onto a subset of the feasible region defined by
perturbing the bounds.
(iii) The convergence properties are established under assumptions that are
typical in the analysis of projected-search methods.  In the case of the
projected search active-set method, it is shown that if the iterates
converge to a nondegenerate stationary point, then the optimal active set
is identified in a finite number of iterations.  It follows that once the
optimal active set has been identified, any method in this class will have
the same convergence rate as its unconstrained counterpart.

\medskip
The paper is organized in seven sections. The standard results associated
with a conventional Wolfe line search are reviewed in
Section~\ref{sec:Wolfe-step}.  Analogous results are established for the
quasi-Wolfe search in Section~\ref{sec:qWolfe-step}.
Section~\ref{sec:Proj-search-methods} considers two broad classes of
projected-search methods applied to a constrained optimization problem.
% The proposed framework for projected-search methods is formulated in
% Section~\ref{sec:Active-set-methods}.
The convergence properties of the methods are established in
Section~\ref{sec:convergence}.  Section~\ref{sec:numerical} concerns the
numerical performance of the methods.  Comparisons with the
state-of-the-art package \LBFGSB{} are also presented for calibration
purposes.  Section~\ref{sec:Summary} concludes with a summary and
conclusions.

\subsection{Notation}
%-------------------------------------------------------
%% Given vectors $x$ and $y$, the vector consisting of $x$ augmented by $y$ is
%% denoted by $(x,y)$.
The vectors $e$ and $e_j$ denote, respectively, the column vector of ones
and the $j$th column of the identity matrix $I$. The dimensions of $e$,
$e_j$ and $I$ are defined by the context.  The subscript $i$ is appended to
vectors to denote the $i$th component of that vector, whereas the subscript
$k$ is appended to a vector to denote its value during the $k$th iteration
of an algorithm, e.g., $x_k$ represents the value for $x$ during the $k$th
iteration, whereas $\elem{x_k}_i$ denotes the $i$th component of the vector
$x_k$.  The $i$th component of the gradient of the scalar-valued function
$f$ is denoted by $\Grad_i f(x)$.
%% Given vectors $a$ and $b$ with the same dimension, vectors with $i$th
%% component $a_i b_i$ and $a_i/ b_i$ are denoted by $a\dottimes b$ and
%% $a\dotdiv b$ respectively.  Similarly, $\min(a,b)$ is a vector with
%% components $\min(a_i,b_i)$.
The vector two-norm and its induced matrix norm are denoted by
$\norm{\cdot}$.

\section{The Wolfe Line Search} \label{sec:Wolfe-step}
%------------------------------
A typical Wolfe line search may be viewed as a two-stage process. The first
stage involves the determination of an interval containing a Wolfe step, if
one exists.  The second stage locates a Wolfe step in this interval using
safeguarded polynomial interpolation. If the first stage fails, then the
objective function is necessarily unbounded below.  The key principle that
drives the first stage is that certain conditions may be formulated that
determine if an interval contains a Wolfe step. Much of the discussion in
this section is based on the work of Mor\'e and Sorensen~\cite{MorS84},
Mor\`e and Thuente~\cite{MorT94}.  More information may be found in
Wolfe~\cite{Wol72}.  The schematic description of the line-search algorithm
given in Algorithm~\ref{alg:schematic-Wolfe} below follows that of Nocedal
and Wright~\cite{NocW99}.  In order to simplify the notation we omit the
suffix $k$ and consider the univariate function $\phi(\alpha) = f(x +
\alpha p)$ for fixed vectors $x$ and $p$. With this notation the Wolfe
conditions (\ref{eq:Armijo}) and (\ref{eq:strongWolfe}) may be written in
the form
\[
      \phi (\alpha)  \le \phi(0) + \alpha\cA \phi'(0),      \wordss{and}
 \mod{\phi'(\alpha)} \le \cW \mod{\phi'(0)}.
\]
Much of the theory associated with a Wolfe line search is based on the
properties of the auxiliary function
\[
  \omega(\alpha) = \phi(\alpha) - \big(\phi(0) + \alpha\cA \phi'(0) \big),
     \wordss{with}
  \omega'(\alpha) = \phi'(\alpha) - \cA \phi'(0).
\]
Mor\'e and Sorensen~\cite{MorS84} show that a minimizer of this function at
which $\omega$ is negative satisfies the Wolfe conditions.  An example of a
function $\phi$ and its associated auxiliary function $\omega$ are depicted
in Figure~\ref{figshiftedArmijo}.

\begin{figure}[htb]
\begin{center}
\includegraphics[height=0.3\textheight,width=0.55\textwidth]{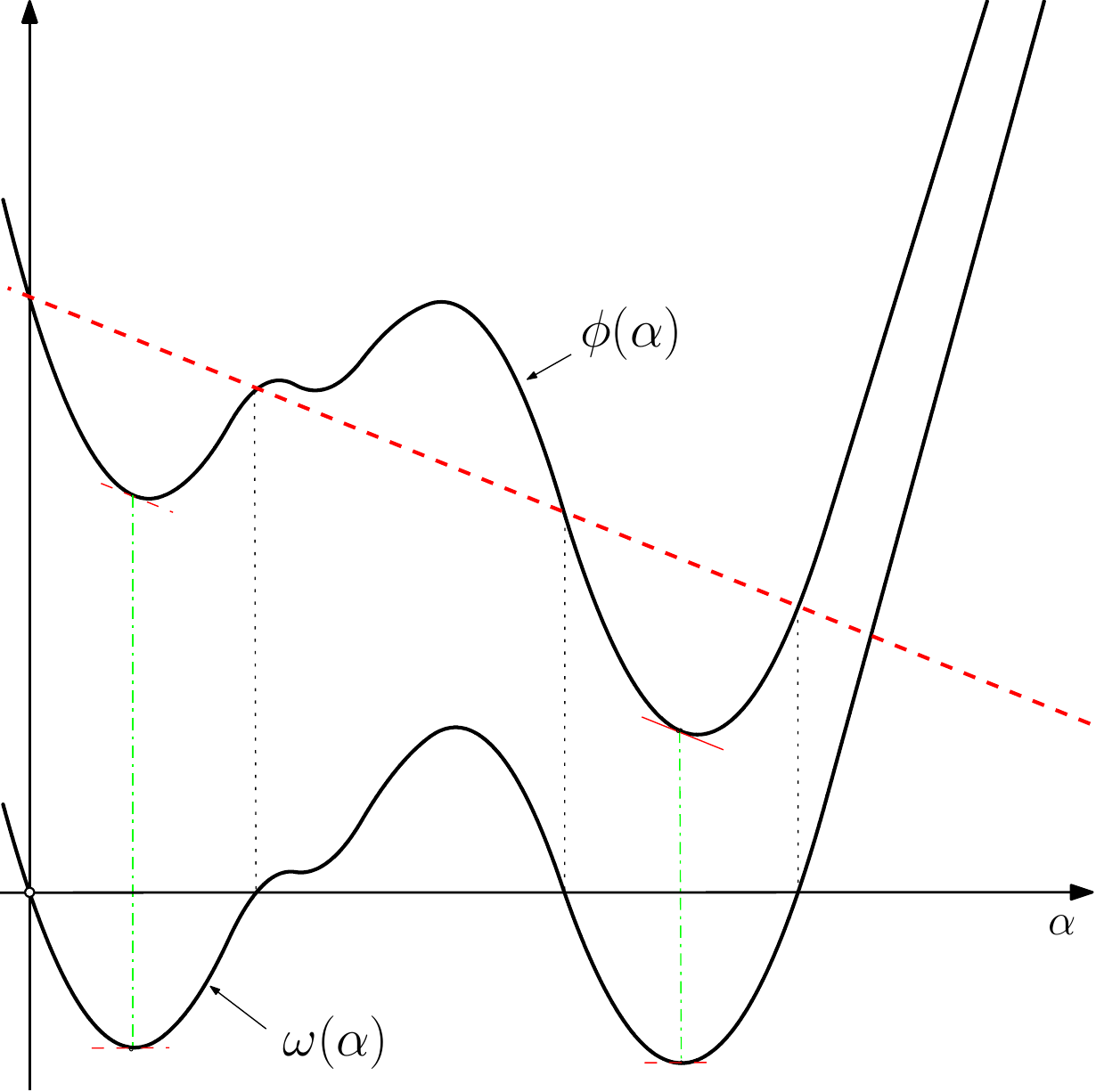}
\parbox[t]{.925\textwidth}{
\caption{\label{figshiftedArmijo}%
\small The graph depicts $\phi(\alpha) = f(x+\alpha p)$ as a function of
       positive $\alpha$, with the shifted function $\omega(\alpha)
       = \phi(\alpha) - \big(\phi(0) + \alpha\cA \phi'(0) \big)$
       superimposed. The dashed line represents the affine function
       $\phi(0) + \alpha\cA \phi'(0)$. }}
\end{center}
\end{figure}

The first stage of a Wolfe line search is motivated by the following
proposition.

\begin{proposition} \label{propUC:Wstage1}
Let $\tSet{\alpha_i}_{i=0}^\infty$ be a strictly monotonically increasing
sequence with $\alpha_0 = 0$. Let $\phi$ and $\omega$ be continuously
differentiable univariate functions such that $\phi'(0) <0$ and
$\omega(\alpha) = \phi(\alpha) - \big(\phi(0) + \alpha\cA \phi'(0) \big)$
with $0< \cA < 1$.  If there exists a least bounded index $j$ such that at
least one of the following conditions is true:
\begin{enumerate}[\bf(a)]
\item\label{Wstage1(a)} $\alpha_j$ is a Wolfe step;

\item\label{Wstage1(b)} $\omega(\alpha_j) \ge \omega(\alpha\jm1)$; or

\item\label{Wstage1(c)} $\omega'(\alpha_j) \ge 0$,
\end{enumerate}
then there exists a Wolfe step $\alphastar \in [\alpha\jm1,\alpha_j]$.
Collectively, {\rm(\ref{Wstage1(a)})--(\ref{Wstage1(c)})} are called the
{\rm stage-one conditions}.
\end{proposition}
\begin{proof}
Observe that $\alpha\jm1$ must satisfy none of the
conditions (\ref{Wstage1(a)})--(\ref{Wstage1(c)}), otherwise $j$ would not
be the least index.
This implies that
$\omega(\alpha\jm1)<\omega(\alpha_{j-2})<\cdot\cdot\cdot<\omega(\alpha_0)=0$
from (\ref{Wstage1(b)}), and $\omega'(\alpha\jm1) < 0$ from (\ref{Wstage1(c)}).

\myparagraph{Case 1}
%-------------------
If (\ref{Wstage1(a)}) is true, the proposition is true trivially.

\myparagraph{Case 2}
%-------------------
If (\ref{Wstage1(b)}) is true, let $\alphabar
= \sup\tSet{\alpha\in[\alpha_{j-1}$, $\alpha_j]\st\omega(\beta)\le 0$ for
all $\beta \in [\alpha_{j-1},\alpha]}$.  If $\alphabar = \alpha_j$, then
$\omega (\alphabar) = \omega(\alpha_j)\ge \omega(\alpha_{j-1})$; if
$\alphabar < \alpha_j$, then by the continuity of $\omega$, $\omega
(\alphabar) = 0 > \omega(\alpha_{j-1})$.  From the mean-value theorem there
must exist an $\alphahat\in (\alpha_{j-1},\alphabar)$ such that
$\omega'(\alphahat) = \big(\omega(\alphabar)
- \omega(\alpha_{j-1})\big)/(\alphabar - \alpha_{j-1}) > 0$.  The function
$\omega(\alpha)$ is continuously differentiable with $\omega'(\alpha\jm1) <
0$ and $\omega'(\alphahat) > 0$. The intermediate-value theorem then
implies that there must exist a step $\alphastar\in
[\alpha_{j-1},\alphahat]$ such that $\omega'(\alphastar)=0$. As
$\omega(\alphastar)\le 0$, $\alphastar$ is a Wolfe step.

%% If (\ref{Wstage1(b)}) is true, let $\alphabar
%% = \sup\tSet{\alpha\in[\alpha_{j-1},\alpha_j]\st\omega(\beta)\le 0$ for all
%% $\beta \in [\alpha_{j-1},\alpha]}$. If $\alphabar = \alpha_j$, then $\omega
%% (\alphabar) = \omega(\alpha_j)\ge \omega(\alpha_{j-1})$; if $\alphabar
%% < \alpha_j$, then by the continuity of $\omega$, $\omega (\alphabar) = 0
%% > \omega(\alpha_{j-1})$.  In either case, by the mean-value theorem, there
%% must exist a step $\alphastar\in [\alpha_{j-1},\alphabar]$ such that
%% $\omega'(\alphastar)=0$. As $\omega(\alphastar)\le 0$, $\alphastar$ is a
%% Wolfe step.

\myparagraph{Case 3}
%-------------------
Finally, consider the case where (\ref{Wstage1(c)}) is true.  If
$\omega(\alpha)<0$ for all $[\alpha_{j-1},\alpha_j]$, then, as
$\omega'(\alpha\jm1) <0$ and $\omega'(\alpha_j)\ge 0$, the continuity of
$\omega'$ and the intermediate-value theorem imply that there exists a step
$\alphastar\in [\alpha_{j-1},\alpha_j]$ such that
$\omega'(\alphastar)=0$. As $\omega(\alphastar) < 0$, $\alphastar$ is a
Wolfe step. Otherwise, if there exists some $\alpha\in
[\alpha_{j-1},\alpha_j]$ such that $\omega(\alpha) \ge 0$, let $\alphabar
= \sup\tSet{\alpha\in[\alpha_{j-1}$, $\alpha_j]\st\omega(\beta)\le 0$ for
all $\beta \in [\alpha_{j-1},\alpha]}$.  The continuity of $\omega$ implies
that $\omega(\alphabar) = 0$. The same argument used in Case~2 may be used
to show that there must exist an $\alphahat\in (\alpha_{j-1},\alphabar)$
such that $\omega'(\alphahat)>0$ and an $\alphastar\in
[\alpha_{j-1},\alphahat]$ such that $\omega'(\alphastar)=0$ with
$\omega(\alphastar)\le 0$.
%% Finally, consider the case where (\ref{Wstage1(c)}) is true.  If
%% $\omega(\alpha)<0$ for all $[\alpha_{j-1},\alpha_j]$, then, since
%% $\omega'(\alpha_j)\ge 0$, the continuity of $\omega'$ implies that there
%% exists a step $\alphastar\in [\alpha_{j-1},\alpha_j]$ such that
%% $\omega'(\alphastar)=0$. As $\omega(\alphastar) < 0$, $\alphastar$ is a
%% Wolfe step.  If there exists some $\alpha\in [\alpha_{j-1},\alpha_j]$ such
%% that $\omega(\alpha) \ge 0$, let $\alphabar
%% = \sup\tSet{\alpha\in[\alpha_{j-1},\alpha_j]\st\omega(\beta)\le 0$ for all
%% $\beta \in [\alpha_{j-1},\alpha]}$. By the continuity of $\omega$, $\omega
%% (\alphabar) = 0$. Since $\omega'(\alpha_{j-1})<0$ and $\omega(\alphabar) =
%% 0 > \omega(\alpha_{j-1})$, by the mean-value theorem, there must exist a
%% step $\alphastar\in [\alpha_{j-1},\alphabar]$ such that
%% $\omega'(\alphastar)=0$.  As $\omega(\alphastar)\le 0$, $\alphastar$ is a
%% Wolfe step.
\end{proof}

\medskip
\noindent
Note that the converse result is not true, e.g., there may be a Wolfe step
in the interval $[0,\alpha_1]$ even though none of the stage-one conditions
are satisfied for $j = 1$.  The behavior of $\omega(\alpha)$ is unknown at
any $\alpha \in (0,\alpha_1)$.

\medskip
If the first step $\alpha_1$ is not a Wolfe step, successively larger steps
are computed until either one of the stage-one conditions is satisfied or
$j$ is such that $\alpha_j = \alpha\submax$.  In practice, $\alpha\submax$
is an upper bound imposed on the step and the search is terminated if the
bound is exceeded during the stage-one iterations.  If a given $\alpha_j$
does not satisfy the stage-one conditions then
$\omega(\alpha_j)<\omega(\alpha\jm1)<\cdot\cdot\cdot<\omega(\alpha_0)=0$.
If the algorithm reaches $\alpha_{j\submax} = \alpha\submax$ and none of the
stage-one conditions have been satisfied, it terminates with
$\alpha_{j\submax}$, which is an Armijo step with the least computed
function value.

Proposition~\ref{propUC:Wstage1} implies that if one of the stage-one
conditions is satisfied at iteration $j$, then the interval
$[\alpha\jm1,\alpha_j]$ must contain a Wolfe step. At this point the line
search terminates successfully if the stage-one condition
(\ref{Wstage1(a)}) is satisfied, or moves on to the second stage.  The
computations associated with the second stage are based on the following
result.

\begin{proposition}\label{propUC:Wstage2}
Let $\phi$ and $\omega$ be defined as in {\rm Proposition~\ref{propUC:Wstage1}}.
Assume there exist distinct points $\alphaLow$ and $\alphaUpp$ such that
\begin{enumerate}[\bf(a)]
\item\label{Wstage2(a)}
  $\omega(\alphaLow)\le 0$;

\item%\label{Wstage2(b)}
  $\omega(\alphaLow) \le \omega(\alphaUpp)$; and

\item\label{Wstage2(c)}
  $\omega'(\alphaLow)(\alphaUpp - \alphaLow) < 0$.
\end{enumerate}
Then there exists a Wolfe step $\alphastar \in \setI$, where $\setI$ is the
interval defined with endpoints $\alphaLow$ and $\alphaUpp$.
% Let $\setI$ be an interval with distinct endpoints $\alphaLow$ and
% $\alphaUpp$ in arbitrary order.
% Let $\phi$ and $\omega$ be continuously
% differentiable univariate functions such that $\omega(\alpha)
% = \phi(\alpha) - \big(\phi(0) + \alpha\cA \phi'(0) \big)$ with $0< \cA <
% \half$.  Assume that $\omega(\alphaLow)\le 0$.
% If $\alphaLow$ and $\alphaUpp$ satisfy the following conditions
% (collectively called the stage-two conditions):
% \begin{enumerate}[\bf(a)]
%
%% \item%\label{Wstage2(1)}
%%   $\omega(\alphaLow) \le \omega(\alphaUpp)$; and
%
% \item%\label{Wstage2(2)}
%   $\omega'(\alphaLow)(\alphaUpp - \alphaLow) < 0$,
% \end{enumerate}
% then there exists a Wolfe step $\alphastar \in \setI$.
\end{proposition}
\begin{proof}
The proof is similar to that of Proposition~\ref{propUC:Wstage1}, and is a
special case of the proof of Proposition~\ref{propUC:qWstage2}.
\end{proof}

\noindent
The conditions (\ref{Wstage2(a)})--(\ref{Wstage2(c)}) of
Proposition~\ref{propUC:Wstage2} are referred to collectively as the
\emph{stage-two conditions}.  The subscripts associated with the points
$\alphaLow$ and $\alphaUpp$ serve to emphasize the fact that
$\omega(\alphaLow)\le \omega(\alphaUpp)$.  It is not necessarily the case
that $\alphaLow < \alphaUpp$.

\medskip
\noindent
Algorithm~\ref{alg:schematic-Wolfe} gives a schematic outline of a Wolfe
line search.  The calculations required for a Wolfe line search may be
organized into two ``functions'' associated with the stage-one and
stage-two conditions.  If the first stage finds an interval that contains a
Wolfe step, the first-stage function labels the endpoints $\alphaLow$ and
$\alphaUpp$ based on relative magnitudes of $\omega(\alpha\jm1)$ and
$\omega(\alpha_j)$, and calls the stage-two function \StageTwo($\alphaLow,
\alphaUpp$).  The second-stage function interpolates the endpoints to
calculate a best-guess step, $\alpha\new$, in the interval.  The
second-stage function is called recursively using $\alpha\new$ and an
existing endpoint, labeling them so that the stage-two conditions hold for
each call.  This is repeated until $\alpha\new$ is a Wolfe step.  In
practice, it rarely takes more than one or two interpolations to find a
Wolfe step.
% Line numbers added by PEG to help compare Wolfe with qWolfe
\begin{algorithm}
\caption{\label{alg:schematic-Wolfe} Schematic outline of a Wolfe line search.\algstrut}
\begin{algorithmic}[1]%[0] for no numbers every line
  \Function{\Wolfe}{$\alpha$\Astrutu}
\State   \textbf{restriction:} $\alpha > 0$;
\State   \textbf{constants:} $\cA \in (0,\half)$, $\cW \in (\cA,1)$, $\gamma_e > 1$,  $\alpha\submax \in (0,+\infty)$; \Astrutu
\State   $\alpha \gets \min\{ \alpha,\alpha\submax\}$; \agap $\alpha\old \gets 0$;
         \While{$\alpha$ is not a Wolfe step  $\AND \alpha \ne \alpha\submax$}
            \If{$\omega(\alpha) \ge \omega(\alpha\old)$}
\State         $\alpha \gets$ \StageTwo($\alpha\old, \alpha$); break;
            \ElsIf{$\omega'(\alpha) \ge 0$}
\State         $\alpha \gets$ \StageTwo($\alpha, \alpha\old$); break;
            \Else
\State         $\alpha\old \gets \alpha$; \agap
               $\alpha \gets \min\Set{ \gamma_e \alpha, \alpha\submax }$; \Comment{Increase $\alpha$ towards $\alpha\submax$}
            \EndIf
         \EndWhile
\State   \Return{$\alpha$};
  \EndFunction
\end{algorithmic}
\begin{algorithmic}[1]%[0] for no numbers every line
  \Function{\StageTwo}{$\alphaLow, \alphaUpp$}
\State   \textbf{restriction:} $\omega(\alphaLow) \le\omega(\alphaUpp)$;
\State   Choose $\alpha\new$ in the interior of the interval defined by $\alphaLow$ and $\alphaUpp$;
         \If{$\alpha\new$ is a Wolfe step}
\State      \Return{$\alpha\new$};
        \ElsIf{$\omega(\alpha\new) \ge \omega(\alphaLow)$}
\State      \Return{\StageTwo($\alphaLow, \alpha\new$)};
        \ElsIf{$\omega'(\alpha\new)(\alphaUpp - \alphaLow) < 0$}
\State      \Return{\StageTwo($\alpha\new, \alphaUpp$)};
        \Else
\State      \Return{\StageTwo($\alpha\new, \alphaLow$)};
        \EndIf
  \EndFunction
\end{algorithmic}
\end{algorithm}

A practical implementation of a Wolfe line search is very complex. There
are many ways to interpolate to obtain a new point in the second stage. The
use of finite precision imposes the need for some sort of safeguarding
during interpolation and gives rise to a whole host of issues, including
how to handle cases when the function or step length are changing by a
value near or less than machine precision. See, e.g., Brent~\cite{Bre73},
Hager~\cite{Hag89}, Ghosh and Hager~\cite{GhoH90}, and Mor\'e and Thuente
\cite{MorT94} for more details.

%\newpage

\section{The Quasi-Wolfe Search}\label{sec:qWolfe-step}
%-------------------------------
As projected-search methods perform a search on the piecewise continuously
differentiable function $f\big(\ProjF(x_k + \alpha p_k)\big)$, it is not
possible for such methods to use a conventional Wolfe line search.  In this
section we consider a new step type, called a \emph{quasi-Wolfe step}, that
is designed to extend the benefits of a Wolfe line search to
projected-search methods.

\subsection{The quasi-Wolfe step} % \label{sec:quasi-Wolfe-step}
%-----------------------------------------------------------------
Performing a search on the univariate function
\[
\f_k(\alpha) = f\big(x_k(\alpha)\big) = f\big(\ProjF(x_k + \alpha p_k)\big),
\]
instead of $\phi_k(\alpha) = f(x_k + \alpha p_k)$, is a substantially more
difficult task because $\f_k$ is only piecewise continuously
differentiable, with a finite number of jump discontinuities in the
derivative (see Section~\ref{subsec:derivatives} below).
Propositions~\ref{propUC:Wstage1} and \ref{propUC:Wstage2}, established in
the preceding section, cannot be used to guarantee a Wolfe step in the
nondifferentiable case because they use the mean-value theorem and require
the line-search function to be differentiable.

In the following discussion, the suffix $k$ is omitted if the iteration
index is not relevant to the discussion.  The definition of a quasi-Wolfe
step involves the left and right derivatives $\fM'(\alpha)$ and
$\fP'(\alpha)$ of $\f$ at $\alpha$, which are defined as
\[
  \fM'(\alpha) = \lim_{\beta \to \alphaM} \f'(\beta) \wordss{and}
  \fP'(\alpha) = \lim_{\beta \to \alphaP} \f'(\beta).
\]

%\medskip
\begin{definition}
Let $\cA$ and $\cW$ be constant scalars such that $0<\cA<\cW<1$.  A step
$\alpha>0$ is called a \emph{quasi-Wolfe step} if it satisfies the
quasi-Armijo condition
\begin{description}
 \item[(\qW1)] $\f(\alpha) \le \f(0) + \alpha\cAd\fP'(0)$,
\end{description}
and at least one of the following conditions:
\begin{description}
 \item[(\qW2)] $\abs{\fM'(\alpha)} \le \cWd \abs{\fP'(0)}$;

 \item[(\qW3)] $\abs{\fP'(\alpha)} \le \cWd \abs{\fP'(0)}$;

 \item[(\qW4)] $\f$ is not differentiable at $\alpha$ and $\fM'(\alpha) \le
   0 \le \fP'(\alpha)$.

\end{description}
\end{definition}

Figure~\ref{figqwConditions} depicts three examples of a kink point
satisfying the quasi-Wolfe conditions.

\begin{figure}[htb]
\begin{center}
\includegraphics[width=4.2cm,height=2.5cm]{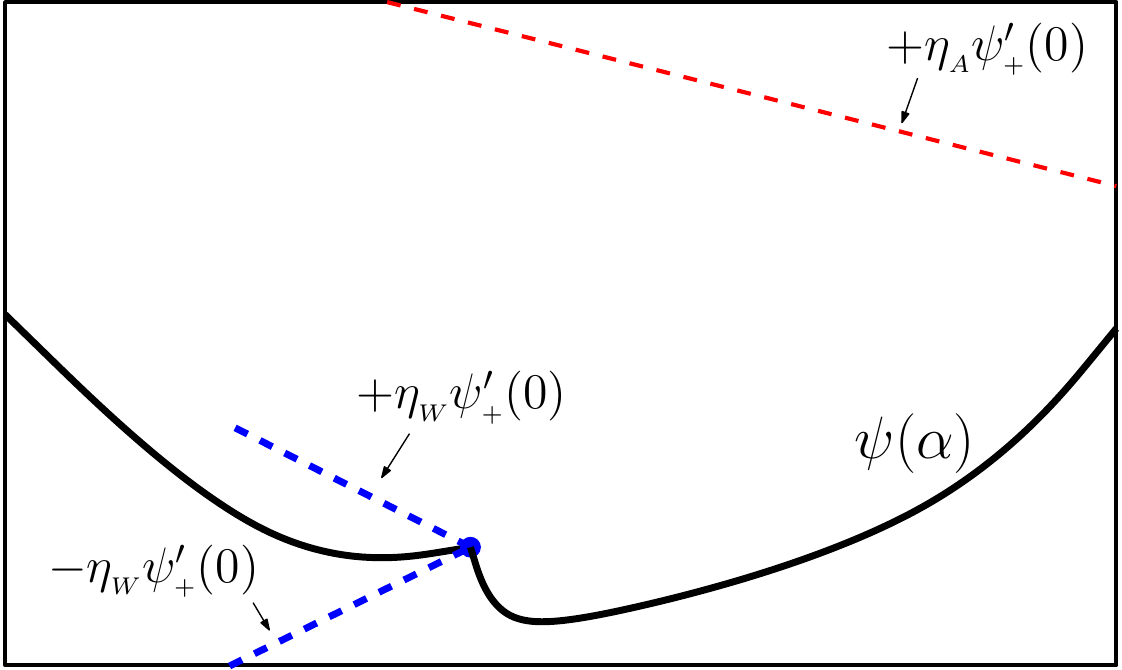}\sgap
\includegraphics[width=4.2cm,height=2.5cm]{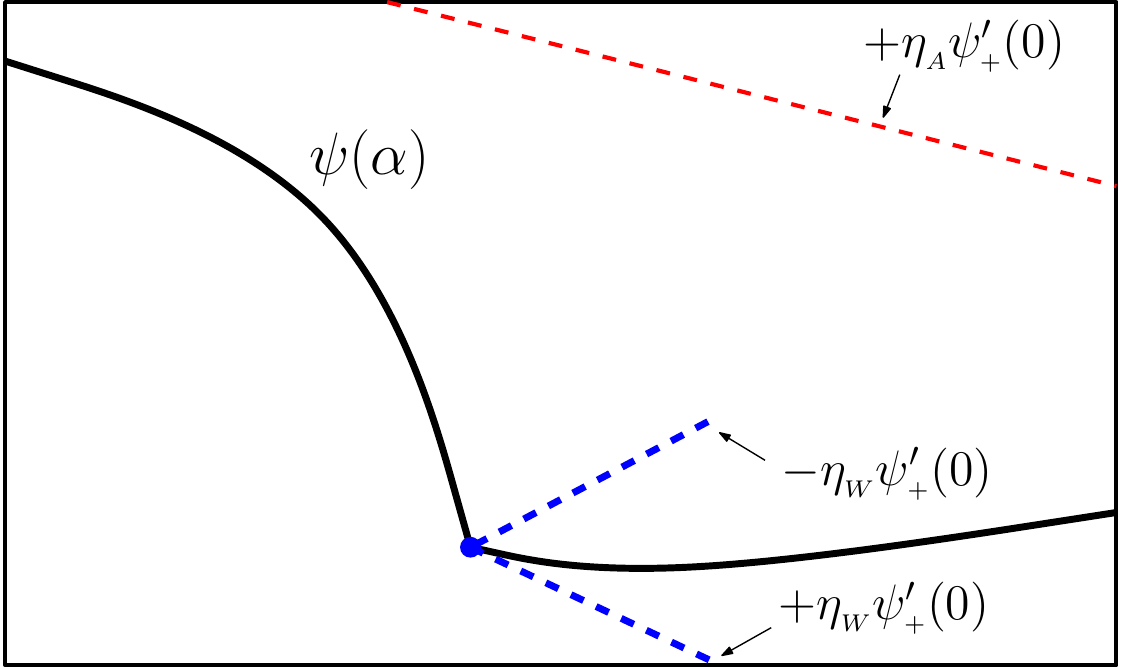}\sgap
\includegraphics[width=4.2cm,height=2.5cm]{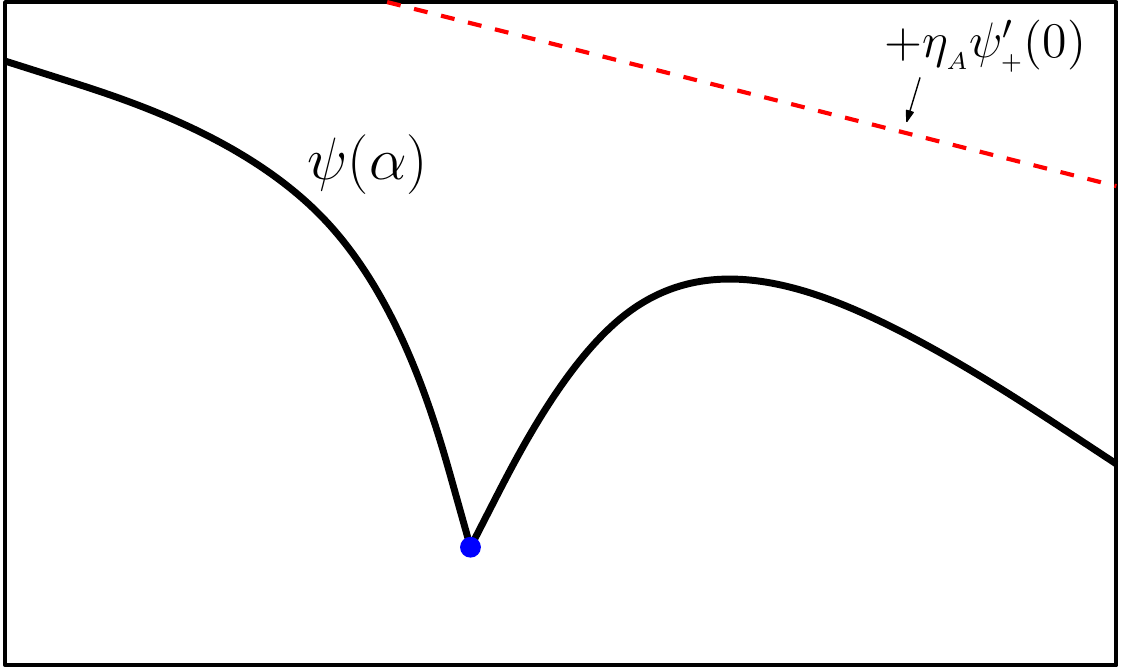}
\parbox[t]{.925\textwidth}{
\caption{\label{figqwConditions}%
\small Three examples of a kink point satisfying the quasi-Wolfe
    conditions.  The left, center and right figures depict kink points
    satisfying conditions {\bf(\qW2)}, {\bf(\qW3)} and {\bf(\qW4)}
    respectively. The slope of each dashed line is marked.}}
\end{center}
\end{figure}

\smallskip
The properties of the new search are characterized by extending the
framework for the differentiable case. In particular, the discussion makes
extensive use of the auxiliary function
\begin{equation} \label{eq:omega-defined}
  \omega     (\alpha) = \f (\alpha) - \big(\f(0) + \alpha\cAd \fP'(0) \big),    \wordss{with}
  \omega'_\pm(\alpha) = \f'_\pm(\alpha) - \cAd \fP'(0).
\end{equation}

The following lemma is used to establish the propositions below.

\begin{lemma}\label{qWlem}
%------------
Let $a$, $b \in \Re$ be such that $0 \le a < b$, and assume that $\theta$
is a univariate, continuous, piecewise continuously differentiable function
with a finite number of jump discontinuities in the derivative.
\begin{enumerate}[\bf(a)]
\item\label{qWlem(a)} If $\thetaP'(a) \le 0$ and $\theta(a) \le \theta(b)$,
then there exists an $x \in (a,b)$ such that
\[ %\begin{equation} \label{eq:(a)}
  \thetaM'(x) \le 0 \le \thetaP'(x).
\] %\end{equation}

\item\label{qWlem(b)}  If $\thetaP'(a) < 0$ and
$\thetaM'(b) > 0$ then there exists an $x \in (a,b)$ such that
\[ %\begin{equation} \label{eq:(b)}
  \thetaM'(x) \le 0 \le \thetaP'(x).
\] %\end{equation}
\end{enumerate}
% elw: lemma is for generic $\theta$ with no reference to $f$
If $\theta$ is differentiable at $x$ then the inequalities in the
conclusions of parts {\rm(\ref{qWlem(a)})} and {\rm(\ref{qWlem(b)})} hold
as equalities.
%\endproof
\end{lemma}
\begin{proof}
  For part~(\ref{qWlem(a)}), let $a = s_0 < s_1 < s_2 < \cdots < s_t <
  s_{t+1} = b$, where $s_1$, $s_2$, \dots, $s_t$ represent all the points
  in $(a,b)$ at which $\theta$ is nondifferentiable. First, suppose that
  $\thetaP'(y)\le 0$ for all $y\in (a,b)$.  Then $\theta$ is
  continuously differentiable and nonincreasing within each subinterval
  $[s_j,s_{j+1}]$ for $j=0$, 1, \dots, $t$.  It follows that
  $\theta(a)\ge\theta(s_1)\ge\cdots\ge\theta(s_t)\ge\theta (b)$.  By
  assumption, this is true only when $\theta (a)=\theta (b)$, which implies
  that $\theta(a) = \theta(s_1)$. Thus, by Rolle's Theorem, there exists an
  $x\in (a,s_1)\subset [a,b]$ such that $\theta'(x)=\theta_\pm'(x)=0$.  Now
  suppose there is a $y\in(a,b)$ such that $\thetaP'(y)>0$, and let $x=\inf
  \Set{ y\in (a,b)\st\thetaP'(y)>0 }$.  Then $x\in (a,b)$, $\thetaP'(x)\ge
  0$, and $\thetaM'(x)=\lim_{y\to x\superminus}\thetaP'(y)\le 0$.  For
  part~(\ref{qWlem(b)}), let $x=\inf\Set{ y\in (a,b) \st \thetaP'(y)>0 }$.
  Then $x\in (a,b)$, $\thetaP'(x)\ge 0$, and $\thetaM'(x)=\lim_{y\to
    x\supminus}\thetaP'(x)\le 0$.
\end{proof}

The next result establishes conditions on $f$ and $\FR$ that guarantee the
existence of a quasi-Wolfe step at each iteration.

\begin{proposition}
%------------------
Let $f$ be a scalar-valued continuously differentiable function defined on
$\FR=\Set{x\in \Re^n \st \bL \le x\le \bU }$.  Assume that $x_0\in\FR$ is
chosen such that the level set $\setL\big(f(x_0)\big)$ is bounded, and
assume that $\tSet{ p_k }$ is a sequence of descent directions. If
$\cA$ and $\cW$ are fixed scalars such that $0< \cAd < \cWd < 1$, then at
every iteration $k$ either there exists an $\alpha\up{k}\L > 0$ and an
interval $(\alpha\up{k}\L,\alpha\up{k}\U)$ such that every
$\alpha\in(\alpha\up{k}\L,\alpha\up{k}\U)$ is a quasi-Wolfe step, or there
exists a quasi-Wolfe step that satisfies the condition {\bf(\qW4)}.
\end{proposition}
\begin{proof}
We omit the suffix $k$ and write $\f(\alpha) = f\big(\ProjF(x + \alpha
p)\big)$.  First, it will be shown that there exists a positive scalar
$\sigma$ such that the function $\omega$ of (\ref{eq:omega-defined})
satisfies $\omega(\alpha)<0$ for all $\alpha\in (0,\sigma)$. As
$\fP'(0)=\Grad f(x)\T p < 0$ and $\cAd < 1$, it must hold that
\[
    \omegaP'(0) = (1 - \cAd)\fP'(0)<0,
\]
in which case there must be a scalar $\sigma$ ($\sigma > 0$) such that
$\omega(\alpha)<0$ for all $\alpha\in (0,\sigma)$.  It follows that there
exists a $\sigma_1\in (0,\sigma)$ such that $\omega(\sigma_1) < 0$.

From the compactness of the level set $\setL\big(f(x_0)\big)$, $\f(\alpha)$ is
bounded below by some constant $\f\low$, i.e., $\f(\alpha)\ge \f\low$ for
all $\alpha\in [0,\infty)$.  As $\f(0) + \alpha\cAd \fP'(0)\to -\infty$ as
$\alpha\to\infty$, there must exist a positive $\sigma_2$ such that
$\f(0)+\sigma_2\drop \cAd\fP'(0)=\f\low$, and we have
\[
  \omega(\sigma_2) = \f(\sigma_2) - \f(0) - \sigma_2\drop\cAd\fP'(0)
                   = \f(\sigma_2) - \f\low  \ge 0.
\]
Given scalars $\sigma_1$ and $\sigma_2$ ($0\le\sigma_1 <\sigma_2$) such
that $\omega(\sigma_1)<0$ and $\omega(\sigma_2)\ge 0$, the intermediate-value
theorem states that there must exist at least one positive $\alpha$ such
that $\omega(\alpha)=0$. Let $\beta$ denote the least positive root of
$\omega(\alpha)=0$, then $\omega(\alpha)<0$ for all $\alpha\in (0,\beta)$.
As $\omega(0)=0$, $\omega(\beta)=0$, and $\omegaP'(0)<0$, by
Lemma~\ref{qWlem} (\ref{qWlem(a)}), there exists an $\xi\in (0,\beta)$ such
that
\[
  \omegaM'(\xi)\le 0\le \omegaP'(\xi),    \words{or, equivalently,}
  \fM'(\xi) \le \cAd\fP'(0) \le \fP'(\xi).
\]

By construction, $\xi\in (0,\beta)$, which implies that $\omega(\xi)\le 0$, or
equivalently, $\xi$ satisfies the quasi-Armijo condition (\qW1).
If $\fP'(\xi)\le 0$, then the inequality $\cA < \cW$
implies that $\xi$ is a quasi-Wolfe step that satisfies the derivative
condition (\qW3).  By the piecewise continuity of $\f'_+(\alpha)$, there
exists an $\alpha\L>0$ and an interval $(\alpha\L,\alpha\U)$ such that
every $\alpha\in (\alpha\L,\alpha\U)$ is a quasi-Wolfe step.  Otherwise, if
$\fP'(\xi) > 0$, then $\xi$ is a quasi-Wolfe step that satisfies the
condition (\qW4).
\end{proof}

The following result is analogous to Proposition~\ref{propUC:Wstage1} and
motivates the first stage of a quasi-Wolfe search.

\begin{proposition}\label{propUC:qWstage1}
%------------------
Let $\tSet{\alpha_i}_{i=0}^\infty$ be a strictly monotonically increasing
sequence with $\alpha_0 = 0$. Let $\f$ be a continuous
piecewise-differentiable univariate function whose derivative has a finite
number of jump discontinuities. Assume that $\fP'(0)<0$ and define
$\omega(\alpha) = \f(\alpha) - \big(\f(0) + \alpha\cAd \fP'(0) \big)$ with
$0< \cA < 1$. If there exists a least bounded index $j$ such that at least
one of the following ``stage-one'' conditions is true:
\begin{enumerate}[\bf(a)]
\item\label{qWstage1(a)} $\alpha_j$ is a quasi-Wolfe step;

\item\label{qWstage1(b)} $\omega(\alpha_j) \ge \omega(\alpha\jm1)$; or

\item\label{qWstage1(c)} $\omegaM'(\alpha_j) \ge 0$,
\end{enumerate}
then there exists a quasi-Wolfe step $\alphastar\in [\alpha\jm1,\alpha_j]$.
\end{proposition}
\begin{proof}
Observe that $\alpha\jm1$ must satisfy none of the conditions
(\ref{qWstage1(a)})--(\ref{qWstage1(c)}), otherwise $j$ would not be the
least index.  This implies that
$\omega(\alpha\jm1)<\omega(\alpha_{j-2})<\cdot\cdot\cdot<\omega(\alpha_0)=0$
from (\ref{qWstage1(b)}), and $\omegaM'(\alpha\jm1) < 0$ from
(\ref{qWstage1(c)}).

The first step is to show that
\begin{equation} \label{qalphaprop1}
  \omegaP'(\alpha\jm1) < 0.
\end{equation}
If $\omega'(\alpha\jm1)$ exists, then
$\omegaP'(\alpha\jm1)=\omegaM'(\alpha\jm1)<0$.  If $\omega'(\alpha\jm1)$
does not exist, then (\ref{qWstage1(c)}) implies that $\omegaM'(\alpha\jm1)
= \fM'(\alpha\jm1) - \cAd \fP'(0) < 0$, in which case $\fM'(\alpha\jm1)<0$
because $\fP'(0)<0$ by assumption.  As (\qW4) cannot hold at $\alpha\jm1$,
it follows that $\fP'(\alpha\jm1)< 0$. Now, if (\qW3) does not hold at
$\alpha\jm1$ then $\fP'(\alpha\jm1) < \cWd \fP'(0) < \cAd \fP'(0)$.  Thus,
$\omegaP'(\alpha\jm1) = \fP'(\alpha\jm1) - \cAd\fP'(0) <0$.  The inequality
(\ref{qalphaprop1}) is used in the proofs that follow.

\myparagraph{Case 1}
%-------------------
If (\ref{qWstage1(a)}) is true, the proposition holds trivially.

\myparagraph{Case 2}
%-------------------
If (\ref{qWstage1(b)}) is true, let $\alphabar
= \sup\tSet{\alpha\in[\alpha_{j-1}$, $\alpha_j]\st\omega(\beta)\le 0$ for
all $\beta \in [\alpha_{j-1},\alpha]}$.  If $\alphabar = \alpha_j$, then
$\omega (\alphabar) = \omega(\alpha_j)\ge \omega(\alpha_{j-1})$; if
$\alphabar < \alpha_j$, then by the continuity of $\omega$, $\omega
(\alphabar) = 0 > \omega(\alpha_{j-1})$.  In either case, as
$\omegaP'(\alpha\jm1) < 0$ by (\ref{qalphaprop1}), part~(\ref{qWlem(a)}) of
Lemma~\ref{qWlem} implies that there exists an $\alphastar \in
[\alpha\jm1,\alphabar]$ such that
\[
 \omegaM'(\alphastar) \le 0 \le \omegaP'(\alphastar).
\]
This implies that
\[
 \fM'(\alphastar) \le \cAd \fP'(0) \le \fP'(\alphastar).
\]
From the definition of $\alphabar$, $\alphastar$ satisfies the quasi-Armijo
condition (\qW1). As $\fM'(\alphastar) < 0$, if $\fP'(\alphastar) \ge 0$,
then $\alphastar$ is a quasi-Wolfe step by (\qW4).  Alternatively, if
$\fP'(\alphastar) < 0$, then
\[
 \cWd \fP'(0) < \cAd \fP'(0) \le \fP'(\alphastar) < 0,
\]
and again, $\alphastar$ is a quasi-Wolfe step by (\qW3).

\myparagraph{Case 3}
%-------------------
Finally, consider the case where (\ref{qWstage1(c)}) is true, i.e.,
$\omegaM'(\alpha_j) \ge 0$. By (\ref{qalphaprop1}),
$\omegaP'(\alpha\jm1)<0$.  If $\omega(\alpha) \le 0$ for all $\alpha \in
[\alpha\jm1,\alpha_j]$, then either $\omegaM'(\alpha_j) = 0$ such that
$\alpha_j$ is a quasi-Wolfe step, or part~(\ref{qWlem(b)}) of
Lemma~\ref{qWlem} establishes the existence of a step $\alphastar \in
(\alpha\jm1,\alpha_j)$ such that
\[
  \omegaM'(\alphastar) \le 0 \le \omegaP'(\alphastar),
\]
and $\alphastar$ satisfies the quasi-Armijo condition (\qW1).
Otherwise, let $\alphabar = \sup\tSet{\alpha\in[\alpha_{j-1}$,
$\alpha_j]\st\omega(\beta)\le 0$ for all $\beta \in
[\alpha_{j-1},\alpha]}$.  By the continuity of $\omega$, $\omega
(\alphabar) = 0 > \omega(\alpha_{j-1})$.  It follows from
part~(\ref{qWlem(a)}) of Lemma~\ref{qWlem} that there exists a step
$\alphastar \in [\alpha\jm1,\alphabar]$ such that
\[
  \omegaM'(\alphastar) \le 0 \le \omegaP'(\alphastar),
\]
and $\alphastar$ satisfies the quasi-Armijo condition (\qW1).
The same argument used for the preceding case shows that $\alphastar$ is a
quasi-Wolfe step.
\end{proof}

The second stage of a quasi-Wolfe search is based on the following
proposition.

\begin{proposition}\label{propUC:qWstage2}
%------------------
Let $\f$ and $\omega$ be defined as in {\rm Proposition~\ref{propUC:qWstage1}}.
Assume there exist distinct points $\alphaLow$ and $\alphaUpp$ such that
\begin{enumerate}[\bf(a)]
\item%\label{qWstage2(a)}
  $\omega(\alphaLow)\le 0$;

\item%\label{qWstage2(b)}
$\omega(\alphaLow) \le \omega(\alphaUpp)$; and

\item%\label{qWstage2(c)}
$\omegaP'(\alphaLow) < 0$ if $\alphaLow < \alphaUpp$ or
$\omegaM'(\alphaLow) > 0$ if $\alphaLow > \alphaUpp$,
\end{enumerate}
then there exists a quasi-Wolfe step $\alphastar \in \setI$, where $\setI$
is the interval defined with endpoints $\alphaLow$ and $\alphaUpp$.
\end{proposition}
\begin{proof}
% The proof is similar to that of Proposition~\ref{propUC:qWstage1}.
First, consider the case where $\alphaLow < \alphaUpp$.  Let $\alphabar =
\sup\Set{\alpha\in\setI\st\omega(\beta)\le 0$ for all $\beta \in
  [\alphaLow,\alpha]}$.  By the continuity of $\omega$, $\omega(\alphabar)
= 0 \ge \omega(\alphaLow)$.  It follows from part~(\ref{qWlem(a)}) of
Lemma~\ref{qWlem} that there exists a step $\alphastar \in
[\alphaLow,\alphabar]$ such that $\omega(\alphastar)\le 0$ and
\[
  \omegaM'(\alphastar) \le 0 \le \omegaP'(\alphastar).
\]
The same argument used in Proposition~\ref{propUC:qWstage1} shows that
$\alphastar$ is a quasi-Wolfe step.

For the case $\alphaLow > \alphaUpp$, let
$\omegatilde(\alpha)=\omega(\alphaLow+\alphaUpp-\alpha)$.  Then
$\omegatilde(\alphaUpp)=\omega(\alphaLow)\le 0$, and
$\omegatilde\subplus'(\alphaUpp)=-\omegaM'(\alphaLow)<0$.  Let $\alphabar =
\sup \Set{\alpha\in\setI\st\omegatilde(\beta)\le 0$ for all $\beta
  \in[\alphaUpp,\alpha]}$.  The continuity of $\omegatilde$ implies that
$\omegatilde(\alphabar)=0\ge\omegatilde(\alphaUpp)$.  It follows from
part~(\ref{qWlem(a)}) of Lemma~\ref{qWlem} that there exists a step
$\betastar \in [\alphaUpp,\alphabar]$ such that $\omegatilde(\betastar)\le
0$ and
\[
  \omegatilde\subminus'(\betastar) \le 0 \le\omegatilde\subplus'(\betastar).
\]
Let $\alphastar=\alphaLow+\alphaUpp-\betastar$, then $\alphastar\in\setI$,
$\omega(\alphastar)\le 0$ and
\[
  \omegaM'(\alphastar) = - \omegatilde\subplus'(\betastar)\le 0 \le
  -\omegatilde\subminus'(\betastar) =\omegaP'(\alphastar).
\]
It follows that $\alphastar$ is a quasi-Wolfe step.
%$\omegatilde(\alphaUpp)\le\omega(\alphaUpp) = \omegatilde(\alphaLow)$.
\end{proof}

%\section{Practical Issues}\label{sec:practical}
%-------------------------
Although the implementation of a quasi-Wolfe search is similar to that of a
Wolfe line search, there are a number of crucial practical issues
associated with the potential nondifferentiability of the line-search
function.  These issues include the definition of the derivatives of the
line-search function and the computation of a new estimate of a quasi-Wolfe
step.

\subsection{Derivatives of the search function}\label{subsec:derivatives}
%----------------------------------------------
The purpose of this section is to establish expressions for the left- and
right-derivatives of the search function $\f(\alpha) = f\big(x(\alpha)\big)$,
where $x(\alpha)$ is the vector $\ProjF(x+\alpha p)$ with components
\[ %\begin{equation}\label{eq:x}
x_i(\alpha)
 = \begin{cases}
     \bL_i            & \lwords{if} x_i + \alpha p_i < \bL_i, \\
     \bU_i            & \lwords{if} x_i + \alpha p_i > \bU_i, \\
     x_i + \alpha p_i & \lwords{if} \bL_i \le x_i + \alpha p_i \le \bU_i.
   \end{cases}
\] %\end{equation}
First, we consider the derivatives of $x(\alpha)$.  Under the assumptions
that $x$ is feasible and $\alpha$ is positive,
it must hold that if $x_i+\alpha p_i < \bL_i$ then $p_i<0$, and if
$x_i+\alpha p_i > \bU_i$, then $p_i>0$.  This implies that the right
derivative of $x(\alpha)$ with respect to $\alpha$ is given by
\[ %\begin{equation}\label{eq:xrprime}
 [x\subplus'(\alpha)]_i
    = \begin{cases}
        0    & \lwords{if} x_i(\alpha) = \bL_i \words{and} p_i < 0, \\
        0    & \lwords{if} x_i(\alpha) = \bU_i \words{and} p_i > 0, \\
        p_i  & \lwords{otherwise.}
      \end{cases}
\] %\end{equation}
The vector $x\subplus'(\alpha)$ may be expressed in terms of $\proj_x(p)$,
the \emph{projected direction of $p$ at $x$}, which is defined as
\[ %\begin{equation}\label{eq:projdirection}
  \elem{\proj_x(p)}_i
    = \begin{cases}
        0   & \lwords{if} x_i = \bL_i \words{and} p_i < 0, \\
        0   & \lwords{if} x_i = \bU_i \words{and} p_i > 0, \\
        p_i & \lwords{otherwise.}
      \end{cases}
\] %\end{equation}
The vector $\proj_x(p)$ represents the projection of $p$ onto the closure
of the set of feasible directions at $x(\alpha)$. If $x(\alpha)$ is
differentiable at a point $\alpha$, then
\begin{equation}\label{eq:xprime}
  x'(\alpha)=x\subplus'(\alpha)=P_{x(\alpha)}(p).
\end{equation}

If $x(\alpha)$ is not differentiable at $\alpha$ then there must be at
least one index $i$ such that
\[
  (x_i + \alpha p_i = \bL_i \words{and} p_i < 0) \wordss{or}
  (x_i + \alpha p_i = \bU_i \words{and} p_i > 0).
\]
An $\alpha$ satisfying one of these conditions is called a \emph{kink step
  with respect to $i$} and it also must hold that
$x\subplus'(\alpha) \ne x\subminus'(\alpha)$.  In order to compute the left
derivative $x\subminus'(\alpha)$, consider the values of $x'(\beta)$ as
$\beta$ approaches $\alpha$ from below. If $\alpha$ is a kink step with
respect to $i$ then $x_i + \beta p_i$ is feasible for all $\beta$
sufficiently close to $\alpha$ and it follows from (\ref{eq:xprime}) that
$x_i'(\beta) = p_i$. If this value is combined with the components of
$x_i'(\beta)$ associated with the differentiable case, we obtain
$x\subminus'(\alpha) = P\supminus_{x(\alpha)}(p)$,
where
\[
 \elem{P\supminus_{x(\alpha)}(p)}_i
  = \begin{cases}
      p_i                        & \lwords{if} \alpha \words{is a kink step with respect to} i, \\
      \elem{P_{x(\alpha)}(p)}\id & \lwords{otherwise.}
    \end{cases}
\]

Next we consider the derivatives of the search function $\f(\alpha)$. If
$\f(\alpha)$ is differentiable at $\alpha$, then the chain rule gives
\[
 \f'(\alpha)
  = \frac{d}{d\alpha} f\big(x(\alpha)\big)
  = \Grad f\big(x(\alpha)\big){}^T \frac{d}{d\alpha}x(\alpha)
  = \Grad f\big(x(\alpha)\big){}^T x'(\alpha).
\]
Using this expression with the expression \eqref{eq:xprime} for
$x'(\alpha)$ gives
\[
  \f'(\alpha) = \Grad f\big(x(\alpha)\big){}\T P_{x(\alpha)}(p).
\]

If $\f(\alpha)$ is not differentiable at $\alpha$, then $\alpha$ is a kink
step and $\fM'(\alpha) \ne \fP'(\alpha)$. For any $\alpha$, $\lim_{\beta
  \to \alphaP} x'(\beta)=x\subplus'(\alpha)$, and $\lim_{\beta \to \alphaM}
x'(\beta)=x\subminus(\alpha)$.  It follows that the right- and
left-derivatives of $\fP(\alpha)$ with respect to $\alpha$ are given by
\[
  \fP'(\alpha) = \Grad f\big(x(\alpha)\big){}\T x\subplus'(\alpha)
  = \Grad f\big(x(\alpha)\big){}\T P_{x(\alpha)}(p),
\]
and
\[
  \fM'(\alpha) = \Grad f\big(x(\alpha)\big){}\T x\subminus(\alpha)
  = \Grad f\big(x(\alpha)\big){}\T P\supminus_{x(\alpha)}(p).
\]
These expressions imply that there is a jump of magnitude $\Mod{p_i \Grad_i
  f\big(x(\alpha)\big)}$ in the derivative of $\f$ at a kink step with
respect to $i$.

\subsection{Computing a quasi-Wolfe step}%\label{subsec:estimates}
%----------------------------------------
As in the Wolfe line search discussed in Section~\ref{sec:Wolfe-step}, a
quasi-Wolfe search may be regarded as having two stages. The first stage
begins with an initial step length $\alpha_0$ and continues with steps of
increasing magnitude until one of three things happens: an acceptable step
length is found; an interval that contains a quasi-Wolfe step is found; or
the step is considered to be unbounded.  In practice, the search is
terminated if the computed step length exceeds a preassigned upper bound
$\alpha\submax$ during the first-stage iterations.  If the search
terminates at $\alpha\submax$ without finding an interval containing a
quasi-Wolfe step, then every step computed up to that point satisfies the
quasi-Armijo condition.

If the first stage terminates with a bounded step, the second stage
repeatedly calls a function \StageTwo($\alpha\low$, $\alpha\high$), where
\begin{enumerate}[\bf(a)]
  \item the interval bounded by $\alpha\low$ and $\alpha\high$ contains a
        quasi-Wolfe step;

  \item among all the step lengths generated so far, $\alpha\low$ gives the
        least value of $\omega$;

  \item $\alpha\high$ is chosen so that $\omegaP'(\alpha\low)<0$ if
        $\alpha\low<\alpha\high$, or $\omegaM'(\alpha\low)>0$ if
        $\alpha\low>\alpha\high$.
\end{enumerate}

\noindent
It must be emphasized that in practice, the stage-two calculations are not
implemented as a recursive procedure.  The recursive structure depicted in
Algorithm~\ref{alg:schematic-Wolfe} is illustrative and reflects the fact
that the intervals defined by $\alpha\low$ and $\alpha\high$ form
a \emph{nested sequence}.  If $\setI_0$ is the interval resulting from
stage-one, the computations of stage-two generate a sequence of intervals
$\{ \setI_j \}$ and a sequence of points $\{ \alpha\upj\low \}$ such that
$\alpha\up{j} \in \setI_j$, each $\setI_j$ contains a quasi-Wolfe step, and
$\setI_j\subset \setI_{j-1}$.  The intervals $\setI_j$ form a nested
sequence of ``intervals of uncertainty''.
Algorithm~\ref{alg:schematic-qWolfe} gives a schematic outline of a
quasi-Wolfe search.
% Line numbers added by PEG to help compare Wolfe with qWolfe
\begin{algorithm}
\caption{\label{alg:schematic-qWolfe} Schematic outline of a quasi-Wolfe search.\algstrut}
\begin{algorithmic}[1]%[0] for no numbers every line
  \Function{\quasiWolfe}{$\alpha$\Astrutu}
\State   \textbf{restriction:} $\alpha > 0$;
\State   \textbf{constants:} $\cA \in (0,\half)$, $\cW \in (\cA,1)$, $\gamma_e > 1$, $\alpha\submax \in (0,+\infty)$; \Astrutu
\State   $\alpha \gets \min\{ \alpha,\alpha\submax\}$; \agap $\alpha\old \gets 0$;
         \While{$\alpha$ is not a quasi-Wolfe step  $\AND \alpha \ne \alpha\submax$}
           \If{$\omega(\alpha) \ge \omega(\alpha\old)$}
\State           $\alpha \gets$ \StageTwo($\alpha\old,\alpha$); break;
           \ElsIf{$\omegaM'(\alpha) \ge 0$}
\State          $\alpha \gets$ \StageTwo($\alpha,\alpha\old$); break;
           \Else
\State       $\alpha\old \gets \alpha$; \agap
               $\alpha \gets \min\Set{ \gamma_e \alpha, \alpha\submax }$; \Comment{Increase $\alpha$ towards $\alpha\submax$}
           \EndIf
         \EndWhile
\State   \Return{$\alpha$};
  \EndFunction
\end{algorithmic}
\begin{algorithmic}[1]%[0] for no numbers every line
  \Function{\StageTwo}{$\alphaLow, \alphaUpp$}
\State  \textbf{restriction:} $\omega(\alphaLow) \le\omega(\alphaUpp)$;
\State  Choose $\alpha\new$ in the interior of the interval defined by $\alphaLow$ and $\alphaUpp$;
        \If{$\alpha\new$ is a quasi-Wolfe step}
\State       \Return{$\alpha\new$};
        \ElsIf{$\omega(\alpha\new) \ge \omega(\alphaLow)$}
\State         \Return{\StageTwo($\alphaLow, \alpha\new$)};
        \ElsIf{$\omegaP'(\alpha\new) < 0$ $\AND \alphaLow < \alphaUpp$}
\State         \Return{\StageTwo($\alpha\new, \alphaUpp$)};
        \ElsIf{$\omegaM'(\alpha\new) > 0$ $\AND \alphaLow > \alphaUpp$}
\State        \Return{\StageTwo($\alpha\new, \alphaUpp$)};
        \Else
\State         \Return{\StageTwo($\alpha\new, \alphaLow$)};
        \EndIf
  \EndFunction
\end{algorithmic}
\end{algorithm}

A major difference between a Wolfe and a quasi-Wolfe search concerns how
interpolation is used to find new steps in the second stage.  Each time
\StageTwo($\alpha\low$, $\alpha\high$) is invoked, a new trial step
$\alpha\new$ is generated. In the differentiable case, $\alpha\new$ is
usually obtained by polynomial interpolation using the value of $\phi$ and
its derivatives at $\alpha\low$ and $\alpha\high$. If the univariate search
function is only piecewise differentiable, there may be kink points between
$\alpha\low$ and $\alpha\high$, in which case a conventional interpolation
approach may not provide a good estimate of a quasi-Wolfe step.  One
strategy to speed convergence in this situation is to search for the kink
step (if it exists) between $\alpha\low$ and $\alpha\high$ that is closest
to $\alpha\low$. This approach is justified by the following argument. If a
new point $\alpha\new$ is not a quasi-Wolfe step, then based on
Proposition~\ref{propUC:qWstage2}, the end points $\alpha\low$ and
$\alpha\high$ are updated to $\alpha\low$ and $\alpha\new$ in two cases:
\begin{description}

\item[Case (1).] $\omega(\alpha\new)\ge \omega(\alpha\low)$;

\item[Case (2).] $\omegaP'(\alpha\new)< 0$ if $\alpha\high<\alpha\low$, or
$\omegaM'(\alpha\new)>0$ if $\alpha\high>\alpha\low$.
\end{description}
In these cases, the new interval bounded by $\alpha\low$ and $\alpha\new$
will not contain a kink step.  In the remaining case:
\begin{description}
\item[Case (3).] $\omegaP'(\alpha\new)\ge 0$ if $\alpha\high<\alpha\low$, or
$\omegaM'(\alpha\new)\le 0$ if $\alpha\high>\alpha\low$,
\end{description}
the new interval will be bounded by $\alpha\high$ and $\alpha\new$, but may
contain kink points. However, the new interval must contain at least one
fewer kink point.

The search for the kink points proceeds as follows.  The first time the
function \StageTwo($\alpha\low$, $\alpha\high$) is invoked, the kink steps
are computed in $O(n)$ floating-point operations (flops) from
\[%\begin{equation}\label{eqkink}
\kappa_i
= \begin{cases}
        (\bU_i - x_i)/p_i  & \lwords{if} p_i > 0, \\
        (\bL_i - x_i)/p_i  & \lwords{if} p_i < 0, \\
        \infty             & \lwords{if} p_i = 0.
 \end{cases}
\]%\end{equation}
As the interval bounded by $\alpha\low$ and $\alpha\high$ contains a
quasi-Wolfe step, only the kink steps within that interval need be stored.
These steps are then sorted in decreasing order within $O(n\log n)$ flops
using a heapsort algorithm (see, e.g., Williams~\cite{Wil64},
Knuth~\cite[Section~5.2.3]{Knu97}).  The kink step closest to $\alpha\low$,
say $\kappa\superstar_1$, is either the smallest or the largest kink step
within the interval of uncertainty, depending on whether $\alpha\low$ is
smaller or greater than $\alpha\high$.  Once $\kappa\superstar_1$ has been
found, the search for $\kappa\superstar_l$ ($l>1$) is made towards
$\alpha\low$ starting at the kink step $\kappa_{l-1}\superstar$ from the
preceding iteration. To prevent the iterations from lingering at
\textbf{Case~(3)} for too long, an upper limit is imposed on the number of
consecutive kink steps as trial steps. If this limit is reached, a new
trial step is generated by bisection.

% Once an interval is found that brackets a quasi-Wolfe step, only the kink
% steps within that interval are needed and a dynamic array may be used to
% save storage.
Once all the kinks in the interval of uncertainty have been
eliminated, conventional polynomial interpolation may be used to generate a
new step length.  However, some care is necessary to choose the appropriate
left or right derivative for use in the interpolation (see
Section~\ref{subsec:derivatives}).

\medskip
If there is just one kink step in the interval of uncertainty, $\alpha\new$
is set to be that kink step. As the number of kink steps in an interval
increases, it becomes more difficult to strike a balance between making
effective use of the knowledge they exist and efficiency; for example, if
an interval contains $10^6$ kink steps, it is not practical to jump to the
middle one and repeat on each subinterval.

\section{Projected-Search Methods}\label{sec:Proj-search-methods}
%---------------------------------

In this section, we focus on two types of projected-search methods that use
the quasi-Wolfe line search. These methods may be broadly categorized as
active-set methods and interior-point methods.

\subsection{Projected-search active-set methods}\label{subsec:Active-set-methods}
%-----------------------------------------------
Given an initial $x_0 \in \FR$, consider the sequence of iterates
$\tSet{x_k}$ that satisfies $x\kp1 = x_k(\alpha_k) = \ProjF(x_k + \alpha_k
p_k)$, where $\alpha_k$ is a quasi-Wolfe step, and $p_k$ is a descent
direction for $f$ at $x_k$. The search direction $p_k$ is based on the
components of a feasible descent direction $d_k$ computed in terms of a
\emph{working set} of indices at $x_k$ such that
\begin{equation}\label{eq:working-set}
\begin{aligned}
 \setW_k &= \Set{ i \st    [x_k]_i \le \bL_i + \epsilon_k \words{and} \Grad_i f(x_k) > 0 \words{or}\\
         &  \hspace{130pt} [x_k]_i \ge \bU_i - \epsilon_k \words{and} \Grad_i f(x_k) < 0 },
\end{aligned}
\end{equation}
where $\epsilon_0$ is a fixed positive parameter $\epsilon$, and
$\epsilon_k = \min \big \{\epsilon, \big\|\ProjMat\km1^T\Grad
f(x\kdm1)\big\|\big\}$ for $k\ge 1$, with $\ProjMat\km1$ the matrix of
columns of the identity matrix of order $n$ associated with the indices in
the complement of $\setW\km1$ in $\Set{ 1$, $2$, \dots, $n }$.  The matrix
$\ProjMat\kdm1\ProjMat\km1^T$ represents the projection $\projWkm1$ with
respect to the set $\setW\km1$, i.e., for any $d\in\Re^n$ it holds that
$\ProjMat\kdm1\ProjMat\km1^T d = \projWkm1(d)$, with
\[ %\begin{equation}\label{eq:projdirectionSet}
  \elem{\projWkm1(d)}_i
    = \begin{cases}
         0    & \lwords{if} i\in     \setW\km1, \\
         d_i  & \lwords{if} i\not\in \setW\km1.
      \end{cases}
\] %\end{equation}
The search direction $p_k$ is defined in terms of any direction $d_k$ such
that $d\kd = \ProjMat\kd\ProjMat_k^T d\kd$, and $\Grad f(x_k)^T d_k < 0$.
Once $d_k$ is determined, the components of $d_k$ are modified if necessary
to give a search direction $p_k$ such that $\elem{p_k}_i=
\max\{\elem{d_k}_i,0\}$ if $\elem{x_k}_i \le \bL_i+\epsilon_k$ and
$\elem{p_k}_i=\min\{\elem{d_k}_i,0\}$ if $\elem{x_k}_i \ge
\bU_i-\epsilon_k$.  This additional step guarantees convergence in the
situation where iterates approach a boundary point from the interior of the
feasible region---a phenomenon known as zigzagging or jamming (see
Bertsekas~\cite{Ber82}).  The vector $p_k$ satisfies $p\kd =
\ProjMat\kd\ProjMat_k^T p\kd$, and retains the descent property of $d_k$.
For example, if $\elem{d_k}_i \ne 0$ and $\elem{x_k}_i \le \bL_i +
\epsilon_k$, then the definition of $\setW_k$ implies that $\Grad_i f(x_k)
\le 0$.  If $\elem{p_k}_i >0$ then $\elem{p_k}_i =\elem{d_k}_i$.  Otherwise,
$\elem{d_k}_i < 0$ with $\Grad_i f(x_k)\elem{d_k}_i \ge 0$, and setting
$\elem{p_k}_i =0$ makes the directional derivative more negative.

The working set at $x_k$ is a subset of the \emph{extended active set},
which is defined as
\[
  \setA_{\epsilon_k}(x_k)
    = \Set{ i \st \elem{x_k}_i \le \bL_i + \epsilon_k \words{or} \elem{x_k}_i\ge \bU_i-\epsilon_k }.
\]
It is shown in Section \ref{sec:convergence} that, under certain
conditions, $\tSet{\epsilon_k}\to 0$, and
$\setA_{\epsilon_k}(x_k)=\setA(x_k)$ for $k$ sufficiently large, which
would imply that $p_k=d_k$ for $k$ sufficiently large.

A general projected-search method based on the proposed framework is
summarized in Algorithm~\ref{alg:projSearch-qWolfe}.
\begin{algorithm}
\caption{\label{alg:projSearch-qWolfe} A class of active-set projected-search methods}
\begin{algorithmic}[0]%[1] for numbers every line
\State   \textbf{constant:} $\epsilon>0$; \Astrutu
\State  Choose $x_0\in\FR$;

\State  Let $\epsilon_0=\epsilon$;                            \agap
           $k=0$;
        \While{not converged}
\State     Determine the working set $\setW_k$  (\ref{eq:working-set});
\State     Compute a feasible descent direction $d_k$ at $x_k$ such that
                $\elem{d_k}_i=0$ if $i\in\setW_k$;
\State     Modify $d_k$ to give a search direction $p_k$:
\State     $\elem{p_k}_i
               = \begin{cases}
                   \max  \{\elem{d_k}_i,0\} &\mbox{if $\elem{x_k}_i \le \bL_i+\epsilon_k$}, \\
                   \min\,\{\elem{d_k}_i,0\} &\mbox{if $\elem{x_k}_i \ge \bU_i-\epsilon_k$},\\
                   \elem{d_k}_i & \mbox{otherwise;}
                 \end{cases}
           $
\State     Compute a quasi-Wolfe step $\alpha_k$; \agap  $x_{k+1}= \ProjF(x_k + \alpha_k p_k)$;
\State     $\epsilon_{k+1}=\min\big\{\epsilon,\big\|\ProjMat_k^T\Grad f(x\kd)\big\|\big\}$;
\State     $k \gets k+1$;
        \EndWhile
\end{algorithmic}
\end{algorithm}
There are various choices for the direction $d_k$. For example, if $d\kd =
-\ProjMat\kd\ProjMat_k^T\Grad f(x\kd)$, then the method is a variant of
projected gradient.  Other choices include computing $d_k$ as the solution
of the subproblem
\begin{equation}  \label{eq:qNquadratic}
     \minimize{d}\sgap \Grad f(x_k)\T d + \half d\T H_k d
\bgap\subject    \sgap d_i = 0 \words{for all} i\in\setW_k,
\end{equation}
where $H_k$ a positive-definite approximation of $\Hess
f(x_k)$.
%The matrix $H_k$ is
%maintained in reduced-Hessian form and is not stored explicitly (for more
%details, see Gill and Leonard~\cite{GilL03}, Ferry~\cite{Fer11}, and Ferry
%\etal~\cite{FerGWZ20b}).
For the numerical experiments presented in Section~\ref{sec:numerical},
$d_k$ was the solution of (\ref{eq:qNquadratic}) with $H_k$ chosen as a
positive-definite limited-memory BFGS approximation of $\Hess f(x_k)$ (see
Ferry~\cite{Fer11} and Ferry \etal~\cite{FerGWZ20b}).

\subsection{Projected-search interior methods} \label{subsec:Interior-point-methods}
%---------------------------------------------
Almost all interior-point methods for problem (\ref{BC}) are based on
minimizing a sequence of unconstrained functions $M(v \given \mu)$
parameterized by a positive scalar $\mu$.  In general, the function $M(v
\given \mu)$ is either not defined or unbounded for some values of the
variables, which implies that the $v$ are subject to implicit bound
constraints during the minimization. If follows that for a given $\mu$, the
problem to be solved has the general form
\begin{gather}
     \minimize{v\in \Re^{n_v}}\mgap M(v \given \mu)
\bgap\subject             \mgap  \bL_v < v < \bU_v, \tag{IPBC}\label{IPBC}
\end{gather}
where $\bL_v$ and $\bU_v$ are fixed $n_v$-vectors of lower and upper bounds on $v$.  For
example, in the case of the log-barrier method, the function $M(v \given
\mu)$ is given by
\[
 f(x) - \sum_{j=1}^n \mu\ln\big(\bLj\big)
      - \sum_{j=1}^n \mu\ln\big(\bUj\big),
\]
so that $v = x$ and the implicit bounds are $\bL < x < \bU$. The format
(\ref{IPBC}) is quite general, and includes both primal and primal-dual
interior methods (see Section~\ref{subsec:results-interior}).

The proposed projected-search line-search method for problem (\ref{IPBC})
generates a sequence of feasible iterates $\Set{v_k}_{k=0}^\infty$ such
that $v\kp1 = \ProjFk(v_k + \alpha_k \dv_k)$, where $\dv_k$ is a descent
direction for $M(v \given \mu)$ and $\ProjFk(v)$ is the projection of $v$
onto the perturbed feasible region
\begin{equation} \label{eq:FRk-defined}
  \FR_k = \Set{ v \st  v_k - \btol(v_k - \bL_v) \le v \le v_k + \btol(\bU_v-v_k) },
\end{equation}
where $\btol$ a fixed positive scalar such that $0 < \btol < 1$. The
quantity $\btol$ may be interpreted as the ``fraction to the boundary''
parameter used in many conventional interior-point methods.

Projected-search interior-point methods have the potential of requiring
fewer iterations than a conventional interior-point method, thereby
reducing the number of times that a search direction need be computed.
Section~\ref{subsec:results-interior} gives numerical results for a
primal-dual projected-search interior method based on the method of Forsgren and
Gill~\cite{ForG98}.

\section{Convergence Analysis} \label{sec:convergence}
% -----------------------------
Next we consider the convergence properties of the two types of
projected-search methods described in
Section~\ref{sec:Proj-search-methods}.

\subsection{Convergence analysis: projected-search active-set methods} %\label{subsec:convergence-projected}
%-----------------------------------------------
In this section we consider the convergence properties of the class of
projected-search active-set methods described in
Section~\ref{subsec:Active-set-methods}.  As an introduction, we
first consider the convergence of a method with a quasi-Armijo search,
which gives a step satisfying the condition (\ref{eq:quasiArmijo}).

\begin{theorem}[Active-set projected search with a quasi-Armijo search] \label{thm:quasiArmijo}
%-----------------------------------------------------------------------
\startNewLine
Let $f$ be a scalar-valued continuously differentiable function defined on
$\FR= \Set{ x\in \Re^n \st \bL \le x \le \bU }$.  Assume that $x_0\in\FR$
is chosen such that the level set $\setL\big(f(x_0)\big)$ is bounded, and
$\tSet{ x_k }$ is defined by $x\kp1 = x_k(\alpha_k)$, where $\alpha_k$ is a
quasi-Armijo step. For an arbitrarily fixed $\epsilon>0$, define
$\epsilon_0=\epsilon$, and
\[
  \epsilon_k = \min\big\{\epsilon,\big\|\ProjMat\km1^T\Grad f(x\kdm1)\big\|\big\}
\]
for $k\ge 1$, where each $\ProjMat_k$ is a matrix with orthonormal columns
that spans the set of projected directions with respect to the working set
$\setW_k$.  If $\tSet{ p_k }$ is a sequence of
descent directions with $\norm{p_k}\le \theta$ for some constant
$\theta$ independent of $k$, $\ProjMat\kd\ProjMat_k^T p\kd = p\kd$ for all $k$, and
the components of $p_k$ satisfy $\elem{p_k}_i \ge 0$ if $\elem{x_k}_i \le
\bL_i+\epsilon_k$, and $\elem{p_k}_i \le 0$ if $\elem{x_k}_i \ge \bU_i -
\epsilon_k$, then
\[
  \lim_{k\to\infty} \Mod{\Grad f(x_k)\T p_k} = 0.
\]
\end{theorem}
\begin{proof}
First, we show that $ \lim_{k\to\infty} \Mod{\Grad f(x_k)\T p_k} = 0$ if
$\liminf_{k\to\infty}\big\|\ProjMat_k^T\Grad f(x\kd)\big\|\ne 0$.
%% We show that
%% \[
%%   \lim_{k\to\infty} \mod{\Grad f(x_k)\T p_k} = 0.
%% \]
Observe that the quasi-Armijo condition (\ref{eq:quasiArmijo}) implies that
$\tSet{f(x_k) }$ is a strictly decreasing sequence. As the set
$\setL\big(f(x_0)\big)$ is bounded, it follows that $\tSet{f(x_k)
}$ converges, with
\[
  0 =  \lim_{k\to\infty} f(x_k)-f(x\kp1)
   \ge \lim_{k\to\infty} \alpha_k \cA \mod{\Grad f(x_k)\T p_k}=0.
\]
The proof is by contradiction. Suppose that $\mod{\Grad f(x_k)\T
p_k} \not\to 0$ as $k\to\infty$, then there must exist some $\epsilonbar>0$
such that $ \mod{\Grad f(x_k)\T p_k} > \epsilonbar$ infinitely often.  Let
$\setG = \Set{ k \st \mod{\Grad f(x_k)\T p_k} > \epsilonbar }$, then it
must be that $\alpha_k \to 0$ for $k\in\setG$.  For all $k\in\setG$, define
the step $\beta_k = \alpha_k/\sigma$. The hypothesis that
$\liminf_{k\to\infty}\big\|\ProjMat_k^T\Grad f(x\kd)\big\|\ne 0$ implies
$\liminf_{k\to\infty} \epsilon_k > 0$.
As $\tSet{\norm{p_k} }$ is uniformly
bounded by $\theta$ and $\liminf_{k\to\infty} \epsilon_k > 0$, there exists
$\kbar$ such that each component of $\beta_k p_k$ satisfies
$\mod{\elem{\beta_k p_k}_i} < \epsilon_k$ for all $k\ge \kbar$ in $\setG$.
The assumptions on components of $p_k$ imply that $\elem{p_k}_i>0$ only if
$\bU_i-\elem{x_k}_i > \epsilon_k$, and $\elem{p_k}_i<0$ only if
$\elem{x_k}_i-\bL_i>\epsilon_k$. It follows that for all $k\ge \kbar$ in
$\setG$, $\bL_i\le\elem{x_k+\beta_kp_k}_i\le \bU_i$ and $\ProjF(x_k
+ \beta_k p_k) = x_k + \beta_k p_k$.

Let $\setGbar$ denote the indices $k\ge \kbar$ of iterations at which a
reduction in the initial step length was necessary, i.e., $\setGbar = \Set{
k \st t_k > 0, \sgap k\in \setG, \sgap k\ge \kbar }$.  Since $\alpha_k$
converges to zero, $\setGbar$ must be an infinite set.  By definition,
\[
  f(x_k+\beta_k p_k)
   =  f\big(\ProjF(x_k + \beta_k p_k)\big)
   >  f(x_k) +\beta_k \cA \Grad f(x_k)\T p_k, \words{for all} k\in\setGbar.
\]
Adding $-\beta_k\Grad f(x_k)\T p_k$ to both sides and rearranging gives
\begin{align}
 f(x_k + \beta_k p_k) - f(x_k) - \beta_k\Grad f(x_k)\T p_k
    &> -\beta_k (1 - \cA)\Grad f(x_k)\T p_k                         \notag \\
    &>  \beta_k (1 - \cA)\epsilonbar, \words{for all} k\in\setGbar. \label{eq:ineq1}
\end{align}
The Taylor expansion of $f(x_k+\beta_k p_k)$ gives
\begin{equation}\label{eq:eqn1}
  f(x_k+\beta_k p_k)-f(x_k)-\beta_k \Grad f(x_k)\T p_k
    = \beta_k\int_0^1\big(\Grad f(x_k+\tau \beta_k p_k)-\Grad f(x_k)\big){}^T p_k\, d\tau.
\end{equation}
If $\dualnorm{\cdot}$ denotes the norm dual to $\norm{\cdot}$,
i.e., $\dualnorm{x}=\max_{v\neq 0}\mod{x\T v}/\norm{v}$, then
\[
  \Mod{\big(\Grad f(x_k+\tau\beta_k p_k)-\Grad f(x_k)\big){}^T p_k}
     \le \dualnorm{\Grad f(x_k + \tau\beta_k p_k)-\Grad f(x_k)}\norm{p_k}.
\]
If this inequality is substituted in (\ref{eq:eqn1}), it then follows from (\ref{eq:ineq1}) that
\begin{align*}
 (1 - \cA)\epsilonbar
   &<   \int_0^1 \big(\Grad f(x_k+\tau\beta_k p_k) - \Grad f(x_k)\big){}^T p_k \,d\tau\\
   &\le \max_{0\le \tau \le 1} \dualnorm{\Grad f(x_k+\tau\beta_k p_k)-\Grad f(x_k)}\norm{p_k},
                             \words{for all} k\in\setGbar.
\end{align*}
The continuity of $\Grad f$ implies that there exists some $\tau_k\in
[0,\beta_k]$ such that
\[
  \max_{0\le \tau \le 1}\dualnorm{\Grad f(x_k+\tau \beta_k p_k) - \Grad f(x_k)}
     = \dualnorm{\Grad f(x_k+\tau_k p_k)-\Grad f(x_k)}.
\]
Then
\begin{equation}\label{eq:ineq2}
  (1 - \cA)\epsilonbar < \dualnorm{\Grad f(x_k+\tau_k p_k)-\Grad f(x_k)}\norm{p_k}.
\end{equation}
However, $\alpha_kp_k\to 0$ implies $\tau_k p_k\to 0$ for $k\in\setG$, and
the continuity of $\Grad f$ gives
\[
  \dualnorm{\Grad f(x_k+\tau_k p_k)-\Grad f(x_k)}\to 0.
\]
As $\tSet{\norm{p_k} }$ is uniformly bounded above by $\theta$, the
right-hand side of (\ref{eq:ineq2}) converges to zero, which gives the
required contradiction.

Next it will be shown by contradiction that every convergent subsequence of
$\tSet{\mod{\Grad f(x_k)\T p_k}}$ converges to zero regardless of the value
of $\liminf_{k\to\infty} \big\|\ProjMat_k^T\Grad f(x\kd)\big\|$.  As
$\ProjMat\kd\ProjMat_k^T p\kd = p\kd$ for all $k$,
\begin{equation}\label{eq:projdd}
  \mod{\Grad f(x_k)\T p_k} = \mod{\Grad f(x\kd)^T \ProjMat\kd\ProjMat_k^T p\kd}
\end{equation}
for all $k$. Suppose that there exists a convergent subsequence of
$\tSet{\mod{\Grad f(x_k)\T p_k}}$, say $\tSet{\mod{\Grad f(x_{k_j})\T
    p_{k_j}}}$, that converges to a positive value. Then by
(\ref{eq:projdd}), the sequence $\tSet{\mod{\Grad
    f(x_{k_j}\drop)^T\ProjMat_{k_j}\drop\ProjMat_{k_j}^T p_{k_j}\drop}}$
converges to a positive value. As $\tSet{\norm{p_k}}$ is bounded by a
constant $\theta$,
\[
  \liminf_{j\to\infty} \big\|\ProjMat_{k_j}^T\Grad f(x_{k_j}\drop)\big\| >  0.
\]
Applying the previous arguments to the subsequence
$\tSet{\mod{\Grad f(x_{k_j})^T p_{k_j}}}$ gives
\[
  \lim_{j\to\infty} \mod{\Grad f(x_{k_j})^T p_{k_j}} = 0,
\]
which is a contradiction.

As the level set $\setL\big(f(x_0)\big)$ is bounded,
$\tSet{\mod{\Grad f(x_k)\T p_k}}$ is a bounded sequence. It follows that
\[
  \liminf_{k\to\infty}\mod{\Grad f(x_k)\T p_k}
    = \limsup_{k\to\infty} \mod{\Grad f(x_k)\T p_k} = 0.
\]
Therefore,  $\lim_{k\to\infty} \mod{\Grad f(x_k)\T p_k} = 0$.
\end{proof}

\begin{theorem}[Active-set projected search with a quasi-Wolfe search] \label{thm:quasiWolfe}
%---------------------------------------------------------------------
\startNewLine
Let $f$ be a scalar-valued continuously differentiable function defined on
$\FR=\Set{x\in \Re^n \st \bL \le x\le \bU }$.  Assume that $x_0\in\FR$ is
chosen such that the level set $\setL\big(f(x_0)\big)$ is bounded, and
$\tSet{ x_k }$ is given by $x\kp1 = x_k(\alpha_k)$, where $\alpha_k$ is a
quasi-Wolfe step. For an arbitrarily fixed $\epsilon>0$, define
$\epsilon_0=\epsilon$, and
\[
   \epsilon_k = \min\big\{\epsilon,\big\|\ProjMat\km1^T\Grad f(x\kdm1)\big\|\big\}.
\]
for $k\ge 1$, where each $\ProjMat_k$ is a matrix with orthonormal columns
that spans the set of projected directions with respect to the working set $\setW_k$.
If $\tSet{p_k}$ is a sequence of descent directions with $\norm{p_k}\le \theta$
for some constant $\theta$ independent of $k$,
$\ProjMat\kd\ProjMat_k^T p\kd = p\kd$ for all $k$,
and the components of $p_k$ satisfy $\elem{p_k}_i \ge 0$ if $\elem{x_k}_i
\le \bL_i+\epsilon_k$, and $\elem{p_k}_i \le 0$ if $\elem{x_k}_i \ge
\bU_i-\epsilon_k$, then
\[
  \lim_{k\to\infty}    \mod{\Grad f(x_k)\T p_k}     = 0.
\]
% Furthermore, if $\ProjMat\kd\ProjMat_k^T p\kd = p\kd$, then
% $\lim_{k\to\infty} \sgap\mod{\Grad f(x_k)\T p_k} = 0$.
% Also assume that $\{p_k\}$ is a sequence of feasible descent
% directions with components satisfying $\elem{p_k}_i \ge 0$ if $\elem{x_k}_i
% < \bL_i+\epsilon$, and $\elem{p_k}_i \le 0$ if $\elem{x_k}_i >
% \bU_i-\epsilon$ for some fixed $\epsilon>0$, and $\norm{p_k}\le \theta$ for
% some constant $\theta$ independent of $k$. Then
% \[
%   \lim_{k\to\infty} \mod{\Grad f(x_k)\T p_k} = 0.
% \]
\end{theorem}
\begin{proof}
First, we show that $\lim_{k\to\infty} \mod{\Grad f(x_k)\T p_k} = 0$ if
$\liminf_{k\to\infty}\big\|\ProjMat_k^T\Grad f(x\kd)\big\| \ne 0$.  The
first quasi-Wolfe condition (\qW1) is equivalent to the quasi-Armijo
condition, and the arguments in the proof of Theorem~\ref{thm:quasiArmijo}
may be used to show that $\tSet{f(x_k)}$ is a convergent sequence. This
implies that
\[
  \lim_{k\to\infty} \alpha_k \Grad f(x_k)\T p_k = 0.
\]
The proof is by contradiction.  Suppose that $\mod{\Grad f(x_k)\T p_k}
\not\to 0$ as $k\to\infty$, then there exists some $\epsilonbar>0$ such
that $\mod{\Grad f(x_k)\T p_k} > \epsilonbar$ infinitely often.  Let $\setG
= \Set{ k \st \mod{\Grad f(x_k)\T p_k} > \epsilonbar }$, then it must be
that $\alpha_k \to 0$ for $k\in\setG$. As $\tSet{\norm{p_k}}$ is uniformly
bounded above by $\theta$, $\alpha_k p_k \to 0$ for $k\in\setG$.

If the quasi-Wolfe condition (\qW2) is satisfied, then
\[
  \Grad f\big(x_k(\alpha_k)\big){}\T P_{x_k(\alpha_k)}(p_k)\ge -\cWd \mod{\Grad f(x_k)\T p_k}.
\]
Similarly, if the quasi-Wolfe condition (\qW4) is satisfied, then
\[
 \Grad f\big(x_k(\alpha_k)\big){}\T P_{x_k(\alpha_k)}(p_k)\ge 0 \ge -\cWd \mod{\Grad f(x_k)\T p_k}.
\]
In either case, as $\Grad f(x_k)\T p_k<0$, it must hold that
\[
 \Grad f\big(x_k(\alpha_k)\big){}\T P_{x_k(\alpha_k)}(p_k)-\Grad f(x_k)\T p_k
   \ge (1 - \cW)\mod{\Grad f(x_k)\T p_k}
    >  (1 - \cW)\epsilonbar, \rwords{for $k\in\setG$}.
\]
The application of the  triangle inequality yields
\begin{align}
  0<(1-\cW)\epsilonbar
   &<   \Mod{\Grad f\big(x_k(\alpha_k)\big){}\T P_{x_k(\alpha_k)}(p_k)-\Grad f(x_k)\T p_k}                   \notag\\
   &\le \Mod{\Grad f\big(x_k(\alpha_k)\big){}\T P_{x_k(\alpha_k)}(p_k)-\Grad f(x_k)\T P_{x_k(\alpha_k)}(p_k)}\notag\\
   &    \quad\qquad \null + \Mod{\Grad f(x_k)\T P_{x_k(\alpha_k)}(p_k)-\Grad f(x_k)\T p_k}.                  \label{eq:ineq3}
\end{align}
Let $\dualnorm{\cdot}$ denote the norm dual to $\norm{\cdot}$, then
\begin{multline*}
 \Mod{\Grad f\big(x_k(\alpha_k)\big){}\T P_{x_k(\alpha_k)}(p_k)-\Grad f(x_k)\T P_{x_k(\alpha_k)}(p_k)}\\
   \le  \dualnorm{\Grad f\big(x_k(\alpha_k)\big) - \Grad f(x_k)} \norm{P_{x_k(\alpha_k)}(p_k)}
   \le  \dualnorm{\Grad f\big(x_k(\alpha_k)\big) - \Grad f(x_k)} \norm{p_k}.
\end{multline*}
As $\Grad f$ is continuous and $\norm{p_k}$ is uniformly bounded, the
right-hand side of this inequality must converge to zero for $k\in\setG$,
which implies that
\[
  \Mod{\Grad f\big(x_k(\alpha_k)\big)\T P_{x_k(\alpha_k)}(p_k) -\Grad f(x_k)\T P_{x_k(\alpha_k)}(p_k)} \to 0,
          \rwords{for $k\in\setG$}.
\]
Basic norm inequalities give
\begin{align*}
  \Mod{\Grad f(x_k)\T P_{x_k(\alpha_k)}(p_k)-\Grad f(x_k)\T p_k}
   &\le \dualnorm{\Grad f(x_k)}\norm{P_{x_k(\alpha_k)}(p_k)-p_k}\\
   &=   \dualnorm{\Grad f(x_k)}\norm{P_{x_k(\alpha_k)}(p_k)-P_{x_k}(p_k)}.
\end{align*}
As the level set $\setL\big(f(x_0)\big)$ is bounded, and the gradient
$\Grad f$ is continuous, the sequence of dual norms $\tSet{ \dualnorm{\Grad
    f(x_k)} }$ is uniformly bounded. The hypothesis that
$\liminf_{k\to\infty}\big\|\ProjMat_k^T\Grad f(x\kd)\big\|\ne 0$ implies
$\liminf_{k\to\infty} \epsilon_k > 0$.  Also, because
\[
  \norm{x_k(\alpha_k)-x_k}\le \norm{\alpha_k p_k}\to 0, \rwords{for $k\in\setG$},
\]
there must exist an $\kbar$ such that for all $k\ge \kbar$ in $\setG$,
\[
  \elem{x_k(\alpha_k)-x_k}_i < \epsilon_k.
\]
From the assumptions on the components of $p_k$, it must hold that for all
$k\ge \kbar$ in $\setG$, $\elem{p_k}_i<0$ only if $\elem{x_k}_i
> \bL_i+\epsilon_k$, in which case $\elem{x_k(\alpha_k)}_i>\bL_i$; and
$\elem{p_k}_i>0$ only if $\elem{x_k}_i< \bU_i-\epsilon_k$, in which case
$\elem{x_k(\alpha_k)}_i<\bU_i$.  It follows that, for $k\in\setG$
sufficiently large,
\[
  P_{x_k(\alpha_k)}(p\kd) = P_{x_k}(p\kd) = p\kd.
\]
Therefore,
\[
  \dualnorm{\Grad f(x_k)}\norm{P_{x_k(\alpha_k)}(p_k)-P_{x_k}(p_k)}\to 0, \rwords{for $k\in\setG$},
\]
and consequently
\[
  \Mod{\Grad f(x_k)\T P_{x_k(\alpha_k)}(p_k)-\Grad f(x_k)\T p_k}\to 0, \rwords{for $k\in\setG$}.
\]
It follows that the right-hand side of (\ref{eq:ineq3}) converges to zero
for $k\in\setG$, which gives the required contradiction.

It remains to consider the case where the quasi-Wolfe condition (\qW3) is
satisfied, i.e.,
\[
  \Grad f\big(x_k(\alpha_k)\big){}\T P\supminus_{x_k(\alpha_k)}(p_k)\ge -\cW \mod{\Grad f(x_k)\T p_k}.
\]
The assumption that $\Grad f(x_k)\T p_k<0$ gives
\[ %\begin{align*}
\Grad f\big(x_k(\alpha_k)\big){}\T P\supminus_{x_k(\alpha_k)}(p_k)-\Grad f(x_k)\T p_k
   \ge (1 - \cW)\mod{\Grad f(x_k)\T p_k}  >  (1 - \cW)\epsilonbar, \rwords{for $k\in\setG$},
\] %\end{align*}
which implies that
\begin{align}
  0 < (1 - \cW)\epsilonbar
   &<   \Mod{\Grad f\big(x_k(\alpha_k)\big){}\T P\supminus_{x_k(\alpha_k)}(p_k)-\Grad f(x_k)\T p_k}\notag\\
   &\le \Mod{\Grad f\big(x_k(\alpha_k)\big){}\T P\supminus_{x_k(\alpha_k)}(p_k)-\Grad f(x_k)\T P\supminus_{x_k(\alpha_k)}(p_k)}\notag\\
   &  \quad\qquad\null + \Mod{\Grad f(x_k)\T P\supminus_{x_k(\alpha_k)}(p_k)-\Grad f(x_k)\T p_k}.\label{eq:ineq4}
\end{align}
The definition of the dual norm yields
\begin{multline*}
 \Mod{\Grad f\big(x_k(\alpha_k)\big){}\T P\supminus_{x_k(\alpha_k)}(p_k)-\Grad f(x_k)\T P\supminus_{x_k(\alpha_k)}(p_k)} \\
   \le  \dualnorm{\Grad f\big(x_k(\alpha_k)\big) - \Grad f(x_k)} \norm{P\supminus_{x_k(\alpha_k)}(p_k)}
   \le  \dualnorm{\Grad f\big(x_k(\alpha_k)\big) - \Grad f(x_k)} \norm{p_k}.
\end{multline*}
From the continuity of $\Grad f$ and uniform boundedness of $\norm{p_k}$, the
right-hand side of the above inequality converges to zero for $k\in\setG$,
which means that
\[
  \Mod{\Grad f\big(x_k(\alpha_k)\big){}\T P\supminus_{x_k(\alpha_k)}(p_k)
      -\Grad f(x_k)\T P\supminus_{x_k(\alpha_k)}(p_k)} \to 0, \rwords{for $k\in\setG$}.
\]
Also,
\begin{align*}
 \Mod{\Grad f(x_k)\T P\supminus_{x_k(\alpha_k)}(p_k)-\Grad f(x_k)\T p_k}
   & \le \dualnorm{\Grad f(x_k)}\norm{P\supminus_{x_k(\alpha_k)}(p_k)-p_k}\\
   & =   \dualnorm{\Grad f(x_k)}\norm{P\supminus_{x_k(\alpha_k)}(p_k)-P_{x_k}(p_k)}.
\end{align*}
As the level set $\setL\big(f(x_0)\big)$ is bounded, and $\Grad f$ is
continuous, it must hold that the sequence of dual norms
$\tSet{\dualnorm{\Grad f(x_k)}}$ is uniformly bounded. Also, as
\[
  \norm{x_k(\alpha_k)-x_k}\le \norm{\alpha_k p_k}\to 0, \rwords{for $k\in\setG$},
\]
arguments analogous to those used to establish convergence in cases (\qW2)
and (\qW4) give
\[
  P\supminus_{x_k(\alpha_k)}(p\kd) = P_{x_k}(p\kd) = p\kd \rwords{for $k\in\setG$ sufficiently large},
\]
in which case
\[
  \dualnorm{\Grad f(x_k)}\norm{P\supminus_{x_k(\alpha_k)}(p_k)-P_{x_k}(p_k)}\to 0, \rwords{for $k\in\setG$}.
\]
This implies that
\[
  \Mod{\Grad f(x_k)\T P\supminus_{x_k(\alpha_k)}(p_k)-\Grad f(x_k)\T p_k}\to 0, \rwords{for $k\in\setG$}.
\]
It follows that the right-hand side of (\ref{eq:ineq4}) converges to zero
for $k\in\setG$, which gives the required contradiction.

Finally, the same arguments from the proof of Theorem~\ref{thm:quasiArmijo}
imply that
\[
  \lim_{k\to\infty} \mod{\Grad f(x_k)\T p_k} = 0
\]
regardless of the value of $\liminf_{k\to\infty} \big\|\ProjMat_k^T\Grad
f(x\kd)\big\|$.
\end{proof}

Based on the framework described in Section~\ref{subsec:Active-set-methods},
the limit $\lim_{k\to\infty} \mod{\Grad f(x_k)\T p_k} = 0$ implies that
\begin{equation}\label{eq:limDD}
  \lim_{k\to\infty} \mod{\Grad f(x_k)\T d\kd} = 0,
\end{equation}
which would further imply that the projected gradient,
$\ProjMat\kd\ProjMat_k^T\Grad f(x\kd)$, converges to zero for an
appropriate choice of $d_k$.  For example, if $d\kd =
-\ProjMat\kd\ProjMat_k^T\Grad f(x\kd)$, or $d_k$ is the solution of the
subproblem (\ref{eq:qNquadratic}) with the two-norm of the projected
approximate Hessian, $\norm{\ProjMat_k^T H\kd \ProjMat\kd}$, uniformly
bounded, then it may be verified that (\ref{eq:limDD}) implies that
$\norm{\ProjMat_k^T\Grad f(x_k)}\to 0$.

Under the nondegeneracy assumption defined below, any algorithm based on
the proposed framework for which $\norm{\ProjMat_k^T\Grad f(x\kd)}\to 0$
will identify the optimal active set in a finite number of iterations.

\begin{definition}
A point $\xstar \in \FR$ is a \emph{stationary point} of {\rm(\ref{BC})} if
$\Grad_if(\xstar) = 0$ for $\bL_i < \xstar_i < \bU_i$, $\Grad_i f(\xstar)
\ge 0$ for $\xstar_i = \bL\id$ and $\bL_i < \bU_i$, and $\Grad_i f(\xstar)
\le 0$ for $\xstar_i = \bU\id$ and $\bL_i < \bU_i$.  A stationary point
$\xstar$ is \emph{nondegenerate} if $\Grad_i f(\xstar) > 0$ for $\xstar_i =
\bL\id$ and $\bL_i < \bU_i$, and $\Grad_i f(\xstar) < 0$ for $\xstar_i =
\bU\id$ and $\bL_i < \bU_i$.
\end{definition}

The next result shows that a projected-search method with either a
quasi-Armijo or quasi-Wolfe search will identify the optimal active set in
a finite number of iterations.

\begin{theorem} %\label{thm:active}
%---------------------------------------------------------------------
\startNewLine
In addition to the assumptions of {\rm Theorem~\ref{thm:quasiArmijo}} or
{\rm Theorem~\ref{thm:quasiWolfe}}, assume that $\tSet{x_k}$ converges to a
nondegenerate stationary point $\xstar$.  Consider the extended active set
\[
  \setA_{\epsilon_k}(x_k)
    = \Set{ i \st \elem{x_k}_i \le \bL_i + \epsilon_k \words{or}
                  \elem{x_k}_i \ge \bU_i - \epsilon_k }.
\]
If $\norm{\ProjMat_k^T\Grad f(x_k)}\to 0$, then
$\setA_{\epsilon_k}(x_k)=\setA(x_k)=\setA(\xstar)$ for all $k$ sufficiently
large.
\end{theorem}
%---------------------------------------------------------------------
\begin{proof}
First, we show that $\setA(\xstar)\subset \setA_{\epsilon_k}(x_k)$ for $k$
sufficiently large by contradiction. Assume the opposite is true, then
there exists $i\in\setA(\xstar)$ such that $i\notin\setA_{\epsilon_k}(x_k)$
for an infinite subsequence $\setK$, which implies that $i\notin\setW_k$
for all $k\in\setK$. It follows that
\[
  \mod{\Grad_i f(x_k)}\le \norm{\ProjMat_k^T\Grad f(x\kd)} \rwords{for $k\in\setK$}.
\]
As $f$ is continuously differentiable and $\norm{\ProjMat_k^T\Grad
  f(x\kd)}\to 0$, letting $k\to\infty$ in $\setK$ gives
\[
  \mod{\Grad_i f(\xstar)} = \lim_{k\to\infty, k\in\setK}\mod{\Grad_i f(x_k)}=0.
\]
This contradicts the nondegeneracy of $\xstar$.

Now we show that $\setA_{\epsilon_k}(x_k)\subset\setA(\xstar)$ for $k$
sufficiently large. If $\bL_i = \bU_i$, a simple argument gives $i\in
\setA_{\epsilon_k}(x_k)$ and $i\in\setA(\xstar)$. Consider an index $i$
such that $\bL_i < \bU_i$. From the definition of $\epsilon_k$, the
assumption $\norm{\ProjMat_k^T\Grad f(x\kd)}\to 0$ implies that
$\epsilon_k\to 0$.  Hence, for $k$ sufficiently large, $\bL_i+\epsilon_k <
\bU_i-\epsilon_k$.  If $i\notin\setA(\xstar)$, then $\bL_i <
\elem{\xstar}_i<\bU_i$. As $\tSet{x_k}\to \xstar$ and $\epsilon_k\to 0$,
$\bL_i+\epsilon_k < \elem{x_k}_i < \bU_i-\epsilon_k$ for $k$ sufficiently
large, which implies that $i\notin\setA_{\epsilon_k}(x_k)$.  Therefore, if
$i\notin\setA(\xstar)$, then $i\notin\setA_k(x_k)$, i.e.
$\setA_{\epsilon_k}(x_k)\subset \setA(\xstar)$ for $k$ sufficiently
large. We conclude that $\setA_{\epsilon_k}(x_k) = \setA(\xstar)$ for all $k$
sufficiently large.

It remains to show that $\setA(x_k) = \setA_{\epsilon_k}(x_k)$ for $k$
sufficiently large. Obviously $\setA(x_k)\subset\setA_{\epsilon_k}(x_k)$
for all $k$. It is trivial if $\bL_i = \bU_i$.  Now consider the case where
$\bL_i < \bU_i$. Note that $\tSet{x_k}\to \xstar$ implies
$\lim_{k\to\infty}\norm{x_{k+1} - x_k} = 0$. As $\lim_{k\to\infty}(\bU_i -
\epsilon_{k+1}) - (\bL_i+\epsilon_k) = \bU_i - \bL_i>0$,
$\Mod{\elem{x_{k+1} - x_k}_i} < (\bU_i - \epsilon_{k+1}) -
(\bL_i+\epsilon_k)$ for $k$ sufficiently large. Suppose $k_0$ is such that,
for all $k\ge k_0$, $\setA_{\epsilon_k}(x_k) = \setA(\xstar)$ and
$\Mod{\elem{x_{k+1} - x_k}_i} < (\bU_i - \epsilon_{k+1}) -
(\bL_i+\epsilon_k)$.  The inclusion
$\setA_{\epsilon_k}(x_k)\subset\setA(x_k)$ for all $k\ge k_0$ is
established using a contradiction argument.  Assume that there exists
$i\in\setA_{\epsilon_k}(x_k) = \setA(\xstar)$ for all $k\ge k_0$, but
$i\notin\setA(x_k)$ for some $\kbar\ge k_0$. Then either $\bL_i <
\elem{x_{\kbar}}_i \le \bL_i + \epsilon_{\kbar}$ or $\bU_i -
\epsilon_{\kbar} \le \elem{x_{\kbar}}_i < \bU_i$. If the inequality $\bL_i
< \elem{x_{\kbar}}_i \le \bL_i + \epsilon_{\kbar}$ holds, the definition of
$p_k$ in Algorithm~\ref{alg:projSearch-qWolfe} implies that
$\elem{p_{\kbar}}_i\ge 0$, and it must be the case that $\bL_i <
\elem{x_{\kbar}}_i \le \elem{x_{\kbar+1}}_i$. In addition,
$\Mod{\elem{x_{\kbar+1} - x_{\kbar}}_i} < (\bU_i - \epsilon_{\kbar+1}) -
(\bL_i+\epsilon_{\kbar})$ implies that $\elem{x_{\kbar+1}}_i < \bU_i -
\epsilon_{\kbar+1}$. As $i\in\setA_{\epsilon_{\kbar+1}}(x_{\kbar+1})$, it
must hold that $\bL_i < \elem{x_{\kbar}}_i \le \elem{x_{\kbar+1}}_i \le
\bL_i+\epsilon_{\kbar+1}$.  Inductively, for all $k\ge \kbar$, $\bL_i <
\elem{x_{\kbar}}_i \le \elem{x_k}_i \le \bL_i+\epsilon_k$, which implies
that $\elem{\xstar}_i\ge \elem{x_{\kbar}}_i > \bL_i$.  A similar argument
shows that if $\bU_i - \epsilon_{\kbar} \le \elem{x_{\kbar}}_i < \bU_i$,
then $\elem{\xstar}_i \le \elem{x_{\kbar}}_i < \bU_i$.  It follows that
$i\notin\setA(\xstar)$, which contradicts the assumption that
$i\in\setA_{\epsilon_k}(x_k) = \setA(\xstar)$ for all $k\ge k_0$.
Therefore, $\setA_{\epsilon_k}(x_k)\subset\setA(x_k)$ for all $k\ge k_0$,
which completes the proof.
\end{proof}

A simple example shows that the nondegeneracy of a stationary point is
necessary for identifying the optimal active set in a finite number of
iterations.  Let $f:\Re^2\to\Re$ be given by $f(x) = \frac{1}{5}\norm{x}^2$,
and let $\FR = \Set{ x\in\Re^2 : x\ge 0 }$. For this problem $\xstar = (0,0)^T$
is a degenerate stationary point and the global minimizer of $f$ over
$\FR$. Assume that the step length $\alpha_k\le 1$ for all $k$, and let
$\epsilon = \frac{1}{\sqrt{2}}$. Starting from $x_0 = (1,1)^T$, the
projected-gradient method gives
\[
  x_k = \prod_{j=0}^k(1-\tfrac25 \alpha_j) \pmat{ 1 \\ 1 },  \words{and}
  \epsilon_k
      = \tfrac25 \norm{x_{k-1}}
      = \tfrac{2\sqrt{2}}{5}\,\prod_{j=0}^{k-1}\big(1-\tfrac{2}{5}\alpha_j\big)
\]
for $k\ge 1$. Then $\tSet{x_k}$ converges to the degenerate stationary point
$\xstar$, and
\[
  \elem{x_k}_i
    = \prod_{j=0}^k(1-\tfrac{2}{5}\alpha_j)
    > \tfrac{2\sqrt{2}}{5}\prod_{j=0}^{k-1}\big(1-\tfrac{2}{5}\alpha_j\big)
    = \epsilon_k, \quad i = 1, 2
\]
for all $k\ge 1$. It follows that $\setA_{\epsilon_k}(x_k)=\emptyset$ for all
$k$, although $\setA(\xstar)=\Set{1,2}$.

\subsection{Convergence analysis: projected-search interior methods} %\label{subsec:convergence-interior}
% --------------------------------------------
In this section, we consider the convergence properties of the class of
projected-search interior methods described in
Section~\ref{subsec:Interior-point-methods}.  As the parameter $\mu$ is
fixed in problem (\ref{IPBC}), to simplify the notation, we write
$M(v)=M(v\given \mu)$.  The line search is performed on the univariate
function $\f_k(\alpha) = M\big(\ProjFk(v_k + \alpha_k \dv_k)\big)$.

\begin{theorem}[Interior projected search with a quasi-Armijo search]\label{thm:qArmijo-interior}
%---------------------------------------------------------------------
\startNewLine
Assume that $v_0\in\FR$ is chosen such that the level set
$\setL\big(M(v_0)\big)$ is bounded, and $\Set{v_k}$ is defined by
$v\kp1 = \ProjFk(v_k + \alpha_k \dv_k)$, where $\ProjFk(v)$ is the projection of
$v$ onto the set
\[ %\begin{equation} \label{eq:FRk-defined2}
\FR_k = \Set{ v \st  v_k - \btol(v_k - \bL_v) \le v \le v_k + \btol(\bU_v-v_k) },
\] %\end{equation}
with $\btol$ a fixed positive scalar such that $0 < \btol < 1$, $\dv_k$ is a descent
direction, and $\alpha_k$ is a quasi-Armijo step. Also assume that
$\norm{\dv_k}\le \theta$ for some constant $\theta$ independent of $k$
Then
\[
    \lim_{k\to\infty}\mod{\Grad M(v_k)\T \dv_k}=0.
\]
\end{theorem}
%---------------------------------------------------------------------
\begin{proof}
Observe that the quasi-Armijo condition implies that
$\tSet{M(v_k) }$ is a strictly decreasing sequence. As the set
$\setL\big(M(v_0)\big)$ is bounded, it follows that $\tSet{M(v_k)}$
converges, with
\[
      0 =  \lim_{k\to\infty} M(v_k)-M(v\kp1)
        \ge \lim_{k\to\infty} \alpha_k \cA \mod{\Grad M(v_k)\T \dv_k}=0.
\]
The proof is by contradiction. Suppose that $\mod{\Grad M(v_k)\T\dv_k}
\not\to 0$ as $k\to\infty$, then there must exist some $\epsilonbar>0$ such
that $ \mod{\Grad M(v_k)\T \dv_k} > \epsilonbar$ infinitely often.  Let
$\setG = \Set{ k \st \mod{\Grad M(v_k)\T \dv_k} > \epsilonbar }$, then
$\alpha_k \to 0$ for $k\in\setG$.  Let
$\gamma=\inf_{v\in\setL(M(v_0)),i}\Set{v_i-\comp{\bL_v}_i,\comp{\bU_v}_i-v_i}$. By
the continuity of $M(v)$ in $\FR$, the level set $\setL(M(v_0))$ is
closed. Hence, $\setL(M(v_0))$ is a compact subset of the open set $\FR$,
which implies that $\gamma>0$.  By the quasi-Armijo condition, each $v_k$
lies in $\setL(M(v_0))$. Therefore, for each $i$,
$\comp{\bU_v}_i-[v_k]_i>\gamma,$ and $[v_k]_i-\comp{\bL_v}_i>\gamma$ for
all $k$.  For each $k$, define the step $\beta_k = \alpha_k/\sigma$. As
$\Set{\norm{\dv_k}}$ is uniformly bounded by $\theta$, there exists $\bar
k$ such that each componet of $\beta_k\dv_k$ satisfies
$\mod{\comp{\beta_k\dv_k}_i}<\btol\gamma$ for all $k\ge \bar k$ in
$\setG$. It follows that $v_k+\beta_k\dv_k\in\FR_k,$ which implies
$\ProjFk(v_k + \beta_k \dv_k)=v_k + \beta_k \dv_k$.

The continuity of $\Grad M$ implies that there exists some $\tau_k\in[0,\beta_k]$ such that
   \[
     \max_{0\le \tau \le 1}\dualnorm{\Grad M(v_k+\tau \beta_k \dv_k) - \Grad M(v_k)}
     = \dualnorm{\Grad M(v_k+\tau_k \dv_k)-\Grad M(v_k)}.
   \]
Similar arguments in the proof of Theorem~\ref{thm:quasiArmijo} maybe used
to show that
\begin{equation}\label{eq:ineq6}
     (1 - \cA)\epsilonbar < \dualnorm{\Grad M(v_k+\tau_k\dv_k)-\Grad M(v_k)}\norm{\dv_k}.
\end{equation}
However, $\alpha_k\dv_k\to 0$ implies $\tau_k\dv_k\to 0$ for $k\in\setG$,
and the continuity of $\Grad M$ gives
\[
      \dualnorm{\Grad M(v_k+\tau_k\dv_k)-\Grad M(v_k)}\to 0.
\]
As $\tSet{\norm{\dv_k} }$ is uniformly bounded above by $\theta$, the
right-hand side of (\ref{eq:ineq6}) converges to zero, which gives the
required contradiction.
\end{proof}

\begin{theorem}[Interior projected search with a quasi-Wolfe search]%\label{thm:qWolfe-interior}
%-------------------------------------------------------------------
\startNewLine
Assume that $v_0\in\FR$ is chosen such that the level set
$\setL\big(M(v_0)\big)$ is bounded, and $\Set{v_k}$ is defined by $v\kp1 =
\ProjFk(v_k + \alpha_k \dv_k)$, where $\ProjFk(v)$ is the projection of $v$
onto the set
\[ %\begin{equation} \label{eq:FRk-defined3}
     \FR_k = \Set{ v \st  v_k - \btol(v_k - \bL_v) \le v \le v_k + \btol(\bU_v-v_k) },
\] %\end{equation}
with $\btol$ a fixed positive scalar such that $0 < \btol < 1$, $\dv_k$ is
a descent direction, and $\alpha_k$ is a quasi-Wolfe step. Also assume that
$\norm{\dv_k}\le \theta$ for some constant $\theta$ independent of $k$ Then
\[
     \lim_{k\to\infty}\mod{\Grad M(v_k)\T \dv_k}=0.
\]
\end{theorem}
%---------------------------------------------------------------------
\begin{proof}
The first quasi-Wolfe condition (\qW1) is equivalent to the quasi-Armijo
condition, and the arguments in the proof of Theorem~\ref{thm:qArmijo-interior}
may be used to show that $\tSet{M(v_k)}$ is a convergent sequence. This
implies that
\[
     \lim_{k\to\infty} \alpha_k \Grad M(v_k)\T \dv_k = 0.
\]
The proof is by contradiction.  Suppose that $\mod{\Grad M(v_k)\T
  \dv_k}\not\to 0$ as $k\to\infty$, then there exists some $\epsilonbar>0$
such that $\mod{\Grad M(v_k)\T \dv_k} > \epsilonbar$ infinitely often.  Let
$\setG = \tSet{ k \st \mod{\Grad M(v_k)\T \dv_k} > \epsilonbar }$, then it
must be that $\alpha_k \to 0$ for $k\in\setG$. As $\tSet{\norm{\dv_k}}$ is
uniformly bounded above by $\theta$, $\alpha_k \dv_k \to 0$ for
$k\in\setG$.  Let
$\gamma=\inf_{v\in\setL(M(v_0)),i}\Set{v_i-\comp{\bL_v}_i,\comp{\bU_v}_i-v_i}$.
By the same arguments in the proof of Theorem~\ref{thm:qArmijo-interior},
$\gamma>0$, and for each $i$, $\comp{\bU_v}_i-[v_k]_i>\gamma,$ and
$[v_k]_i-\comp{\bL_v}_i>\gamma$ for all $k$.  Therefore, for $k\in\setG$
sufficiently large such that each component
$\comp{\alpha_k\dv_k}_i<\btol\gamma$,
\[
     P_{v_k(\alpha_k)}(\dv\kd) = P_{v_k}(\dv\kd) = \dv\kd.
\]

Similar arguments in the proof of Theorem~\ref{thm:quasiWolfe} may be used
to obtain the required contradictions for each of the cases where (\qW2),
(\qW3) or (\qW4) holds.
\end{proof}

\section{Numerical Experiments} \label{sec:numerical}
%------------------------------
In this section we illustrate the numerical performance of the two classes
of projected-search method described in
Section~\ref{sec:Proj-search-methods}.  All testing was done on problems
taken from the \CUTEst{} test collection (see Bongartz \etal{}
\cite{BonCGT95} and Gould, Orban and Toint~\cite{GouOT03}).

%As of January~1, 2020, the \CUTEst{} test set contains 147
%bound-constrained problems.

\subsection{Numerical results for an active-set projected search method}\label{subsec:results-active}
%------------------------------------------------------------------------
Numerical results were obtained for an active-set projected search method in
which the direction $d_k$ was computed as the solution of
(\ref{eq:qNquadratic}) with $H_k$ chosen as a positive-definite
limited-memory BFGS approximation of $\Hess f(x_k)$ (see Ferry~\cite{Fer11}
and Ferry \etal~\cite{FerGWZ20b}).  All testing was done on problems taken
from the \CUTEst{} test collection (see Bongartz \etal~\cite{BonCGT95} and
Gould, Orban and Toint~\cite{GouOT03}). As of July~1, 2020, the \CUTEst{}
test set contains 154 bound-constrained problems of the form (\ref{BC}).
Although many problems allow for the number of variables and constraints to
be adjusted in the standard interface format (SIF) data file, our tests used
the default dimensions set in the \CUTEst{} distribution.  This gave
problems ranging in size from \Cute{BQ1VAR} (one variable) to \Cute{WALL100}
(149624 variables).

The practical effectiveness of the quasi-Wolfe search was evaluated by
running two limited-memory quasi-Newton methods, one with a quasi-Wolfe
search and the other with a quasi-Armijo search.  The resulting
implementations, \UBOPTqWolfe\@, and \UBOPTqArmijo{} are based on the Fortran
package \UBOPT{} (see Ferry~\etal~\cite{FerGWZ20b}).  In the quasi-Wolfe
search, the kink steps are sorted in decreasing order in $O(n\log n)$ flops
using a heapsort algorithm (see, e.g., Williams~\cite{Wil64},
Knuth~\cite[Section~5.2.3]{Knu97}), adapted from a Fortran implementation
by Byrd \etal~\cite{ByrLNZ95}.  For \UBOPTqWolfe\@, the Armijo tolerance
$\cA$ was set at $10^{-4}$ and the Wolfe tolerance $\cW = 0.9$. In
\UBOPTqArmijo\@, $\cA = 0.3$.  The scalar $\epsilon$ was set to the machine
precision in the expression for $\epsilon_k$ in the calculation
(\ref{eq:working-set}) of the working set.

In order to provide some measure of the efficiency of the projected-search
method relative to a state-of-the-art method for bound-constrained
optimization, the solvers \UBOPTqWolfe{} and \UBOPTqArmijo{} were compared
with the limited-memory method \LBFGSB{} (Byrd \etal~\cite{ByrLNZ95}, Zhu
\etal~\cite{ZhuBLN97}, and Morales and Nocedal \cite{MorN11}).  All three
solvers were applied to the 154 bound-constrained problems from the
\CUTEst{} test set. The runs were terminated at the first point $x\last$
such that
\begin{enumerate}[\bf(a)]
\item \label{res:opt1}
  $\infnorm{\proj_{x\last}\big(-\Grad f(x\last)\big)} \le 10^{-5}\big(1 + \abs{f(x\last)}\big)$ \mgap and

\item %\label{res:opt2}
  $\abs{f(x\last) - f(x\lastm1)} \le 10^7\epsilon\M\times\max\Set{\abs{f(x\last)},\mod{f(x\lastm1)}, 1 }$; \mgap or

\item \label{res:opt3}
 $\infnorm{\proj_{x_{\scriptscriptstyle T}}\big(-\Grad f(x\last)\big)} < \sqrt{\epsilon\M}$,
\end{enumerate}
where $\epsilon\M$ is the machine precision.  In the first iteration of the
algorithms, only condition~(\ref{res:opt3}) is tested.  A nonoptimal
termination was signaled by the violation of a time limit of 3600 seconds,
a limit of $10^6$ iterations, or an abnormal exit because of numerical
difficulties.

The solver \UBOPTqArmijo{} failed on nine problems, with six failing
because of numerical difficulties
(\Cute{BLEACHNG},
\Cute{BQPGAUSS},
\Cute{BRATU1D},
\Cute{GRIDGENA},
\Cute{RAYBENDL},
\Cute{WALL10}, and
\Cute{WEEDS}).
\UBOPTqWolfe{} failed on six problems, with four failures caused by numerical
difficulties
(\Cute{GRIDGENA},
\Cute{PALMER5E},
\Cute{PROBPENL}, and
\Cute{WALL10}).
\UBOPTqWolfe{} identified problem \Cute{BRATU1D} as being unbounded. For
both solvers, \Cute{CYCLOOCTLS} and \Cute{WALL50} could not be solved
within the one hour time limit.  In the cases of numerical difficulties,
the search algorithms were unable to compute an appropriate step.  We note
that for \UBOPTqWolfe{}, the run for \Cute{PROBPENL} terminated at a
near-optimal point that satisfied condition~\eqref{res:opt1} and
$\infnorm{\proj_{x\last}\big(-\Grad f(x\last)\big)} = 1.99 \times 10^{-7}$.
The solver \LBFGSB{} failed on 16 problems. Seven failures were caused by
numerical difficulties
(\Cute{BQPGAUSS},
\Cute{BRATU1D},
\Cute{GRIDGENA},
\Cute{PALMER5A},
\Cute{PALMER5B},
\Cute{PALMER7A}, and
\Cute{WALL10}),
seven problems exceeded the iteration limit
(\Cute{CHEBYQAD},
\Cute{PALMER1E},
\Cute{PALMER2E},
\Cute{PALMER3E},
\Cute{PALMER4E},
\Cute{PALMER6E}, and
\Cute{PALMER8E}),
and two problems exceeded the time limit
(\Cute{CYCLOOCTLS} and \Cute{WALL50}).
More details of the runs are given by Ferry \etal~\cite{FerGWZ20d}.

The relative performance of the solvers is summarized using performance
profiles (in $\log_2$ scale), which were proposed by Dolan and
Mor\'{e}~\cite{DolM02}.  Let $\setP$ denote a set of problems used for a
given numerical experiment.  For each method $s$ we define the function
$\pi_s \st [0,r\subM]\mapsto\Re\supplus$ such that
\[
  \pi_s(\tau) = \frac{1}{n_p}\Mod{\Set{ p \in \setP \st \log_2(r_{p,s}) \le \tau }},
\]
where $n_p$ is the number of problems in the test set and $r_{p,s}$ denotes
the ratio of the number of function evaluations needed to solve problem $p$
with method $s$ and the least number of function evaluations needed to
solve problem $p$. If method $s$ failed for problem $p$, then $r_{p,s}$ is
set to be twice of the maximal ratio. The parameter $r\subM$ is the maximum
value of $\log_2(r_{p,s})$.
\begin{figure}[h]
  \begin{center}
  \includegraphics[width=8cm,height=7cm]{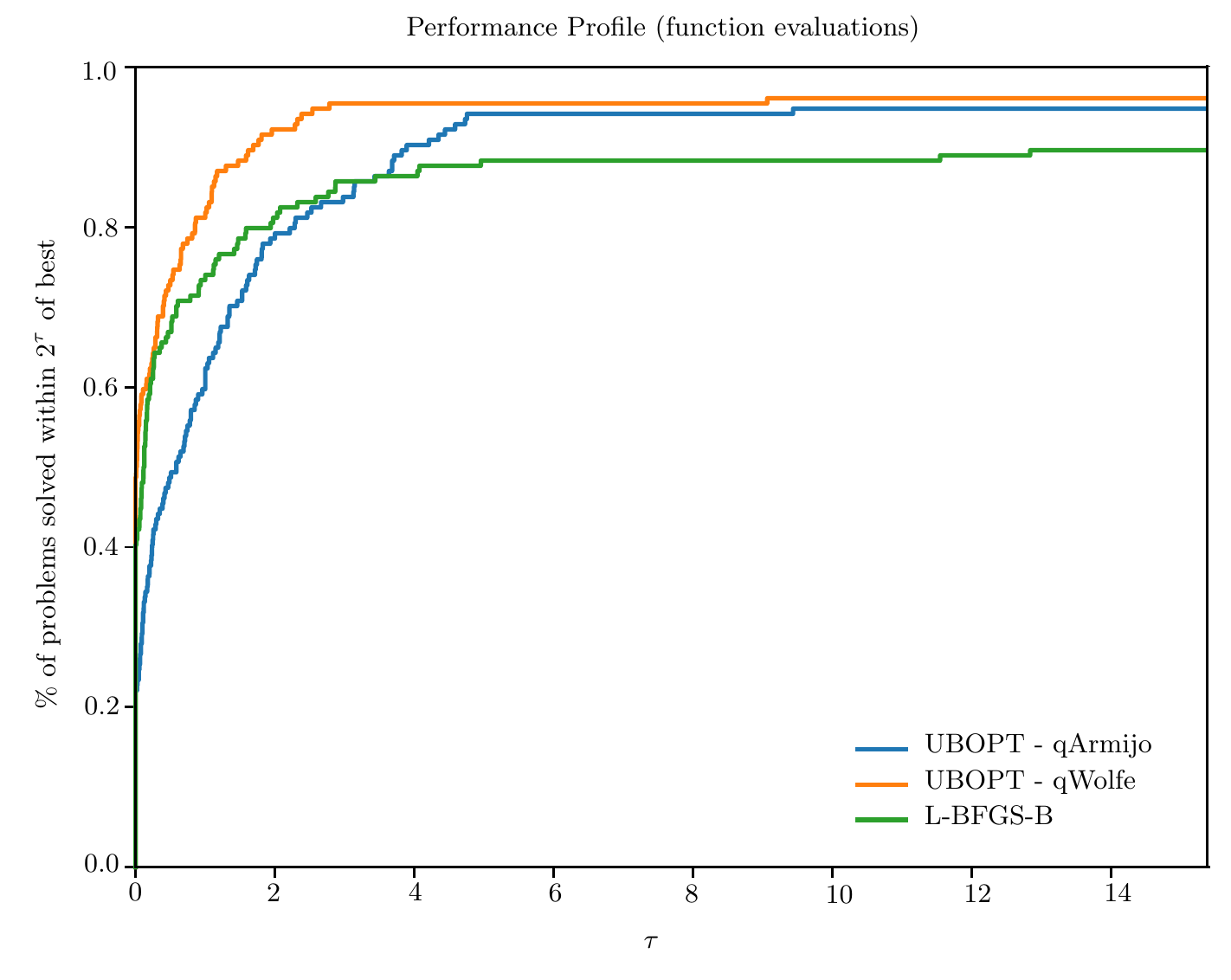}
   \caption{\label{fig:pp-BC-funs}\footnotesize Performance profiles for
     the number of function evaluations required to solve 154
     bound-constrained problems from the CUTEst test set. The figure gives
     the profiles for the three solvers \UBOPTqWolfe\@, \UBOPTqArmijo\@, and
     \texttt{L-BFGS-B}~\cite{ByrLNZ95}.}
  \end{center}
\end{figure}
Figure~\ref{fig:pp-BC-funs} gives the performance profiles for the 154
problems for \UBOPTqWolfe\@, \UBOPTqArmijo\@, and \LBFGSB. The profile
utilized the total number of function evaluations for comparison.
Additional information about the runs used to generate the performance
profiles is given by Ferry \etal~\cite{FerGWZ20d}.  The results indicate
that using a quasi-Wolfe search in \UBOPT{} resulted in a substantially
better performance with respect to function calls than using a quasi-Armijo
search, and comparable and more robust performance with respect to \LBFGSB.

\medskip
A benefit of the Wolfe conditions in the unconstrained case is that the
restriction on the directional derivative guarantees that the approximate
curvature $\big(\Grad f(x\kp1) - \Grad f(x_k)\big){}\T (x\kp1 - x_k)$ is
positive, which is a necessary condition for the quasi-Newton update to
give a positive-definite approximate Hessian.  In the bound-constrained
case, the use of a quasi-Wolfe projected search makes it more likely that
the update can be applied, but it is not possible to guarantee an update in
all cases.  If the next iterate is given by $x\kp1 = \ProjF(x_k + \alpha_k
p_k)$, where $\alpha_k$ is a quasi-Wolfe step, then $\big(\Grad f(x\kp1) -
\Grad f(x_k)\big){}\T (x\kp1 - x_k)$ need not be greater than zero if the
path $\ProjF(x_k + \alpha_k p_k)$ changes direction for some $\alpha \in
(0,\alpha_k)$. If it does change direction, $\fP'(0)$ and $\fM'(\alpha_k)$
may be directional derivatives of $f$ in a direction other than $x\kp1 -
x_k$. This situation is illustrated in Figure~\ref{fig:NoUpdate}, which
depicts a two-dimensional region with lower bounds $x_1 = 0$ and $x_2 =
0$. In this example $\fP'(0)$ is a directional derivative of $f$ in
direction $\elem{p_k}_1$ and $\fM'(\alpha_k)$ is a directional derivative
of $f$ in direction $\elem{p_k}_2$.  As a result, if the path changes
direction for $\alpha \in (0,\alpha_k)$, then there is the possibility that
the quasi-Newton update must be skipped.
\begin{figure}[H]
\begin{center}
\includegraphics[scale=0.5]{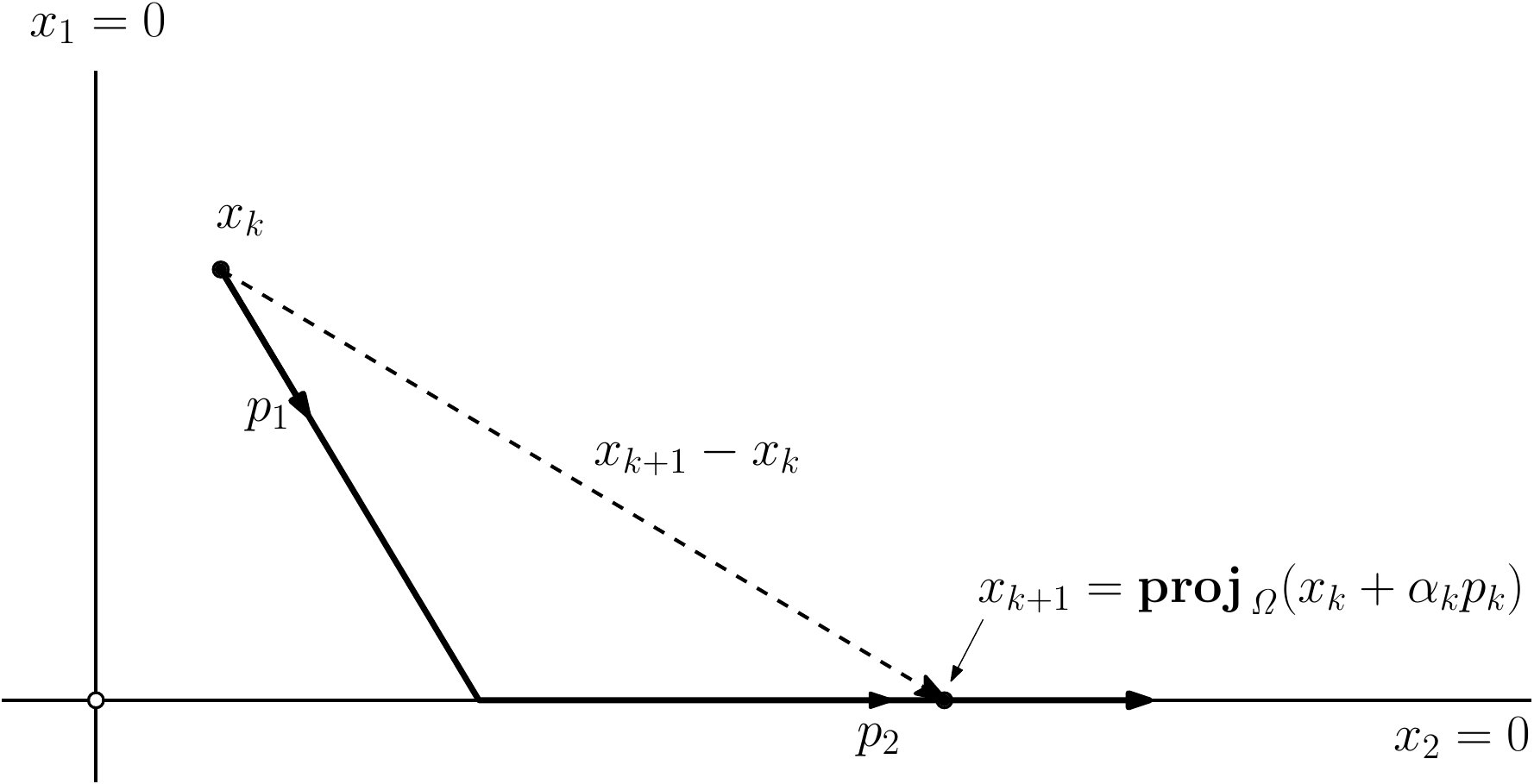}
\parbox[t]{.925\textwidth}{
\caption{\label{fig:NoUpdate} \small Example with no guarantee of an update
  for the approximate Hessian.  }}
\end{center}
\end{figure}

\noindent
It is shown in Section~\ref{sec:convergence} that if $\tSet{x_k}$ converges
to a nondegenerate stationary point, then a quasi-Wolfe search identifies
the active set at the solution in a finite number of iterations. After the
active set stabilizes, a quasi-Wolfe search behaves exactly like a Wolfe
line search in the sense that updates to the approximate Hessian are
guaranteed if $f(x_k + \alpha p_k)$ is bounded below.

To estimate how often the update is likely to be skipped with the
quasi-Wolfe search, statistics were collected from the test problems for
which at least one of the search paths was ``bent'' by projection.  The
application of \UBOPTqWolfe{} resulted in 259 of the potential 637268
updates being skipped ($\approx 0.04\%$). This can be compared to 6537 of
the 679071 updates being skipped ($\approx 1.0\%$) for \UBOPTqArmijo\@. (The
number of updates reflects the number of iterations needed for
convergence.)

\subsection{Numerical results for the projected-search interior method}\label{subsec:results-interior}
%----------------------------------------------------------------------
Numerical results are given for a \Matlab{} implementation of a
projected-search method based on the primal-dual interior method of
Forsgren and Gill~\cite{ForG98}.  For this method, the unconstrained
function
\begin{multline*}
  f(x) - \sum_{j=1}^n \left\{\mu\ln\big(\bLj\big)
                + \mu \ln\big(\zj1(\bLj) \big)
                - \zj1(\bLj)\right\}                      \\
       - \sum_{j=1}^n \left\{\mu\ln\big(\bUj\big)
                + \mu\ln\big(\zj2(\bUj) \big)
                - \zj2(\bUj)\right\}
\end{multline*}
is minimized for a sequence of $\mu$-values such that $\mu \to 0$.  This
implies that the function $M$ of problem (\ref{IPBC}) is given by $M(v
\given \mu) = M(x,z_1,z_2 \given \mu)$, with
\[
 v     = \pmat{ x \\ z_1 \\ z_2 },  \qquad
 \bL_v = \pmat{ \bL \\ 0 \\ 0   },  \wordss{and}
 \bU_v = \pmat{ \bU \\ +\infty \\ +\infty   }.
\]
At any $(x,z_1,z_2)$ such that $\bL < x < \bU$, $z_1 >0$ and $z_2 >0$, let
$X_1=\diag\!\big(x_j - \bL_j\big)$, $X_2=\diag\!\big(\bU_j - x_j\big)$,
$Z_1 = \diag(\zj1)$, and $Z_2 = \diag(\zj2)$.  One iteration of Newton's
method for minimizing $M(x,z_1,z_2 \given \mu)$ requires solving the
equations $\Hess M(v \given \mu) \dv = - \Grad M(v \given \mu)$.  If the
diagonal matrices $\mu X_1\inv $ and $\mu X_2\inv$ in the expression for
$\Hess M(v \given \mu)$ are replaced by $Z_1$ and $Z_2$, we obtain an
approximate Hessian with $n\times n$ principal minor $H_k = \Hess f(x) +
X_1\inv Z\d1 + X_2\inv Z\d2$.  It follows that one iteration of an
\emph{approximate} Newton method for minimizing $M(x,z_1,z_2 \given \mu)$
gives the estimate $(x + \dx$, $z_1+\dz_1$, $z_2 + \dz_2)$, where $\dz_1 =
- X_1\inv\big(\zd1\dottimes(x + \dx - \bL) - \mu e\big)$, $\dz_2 = -
X_2\inv\big(\zd2\dottimes(\bU - x - \dx) - \mu e\big)$, and $\dx$ satisfies
the equations
\begin{equation} \label{eq:ip-Hessian}
 H_k\dx = -  \big(\Grad f(x) -  \mu X_1\inv e + \mu X_2\inv e \big).
\end{equation}
The equations for $(\dx$, $\dz_1$, $\dz_2)$ are equivalent to the
primal-dual path-following equations for problem (\ref{BC}).  Let $v_k$
denote a point such that $\bL_v < v_k < \bU_v$ and let $\dv_k$ denote the
solution of the approximate Newton equations at $v_k$.  If the matrix $H_k$
of (\ref{eq:ip-Hessian}) is positive definite, then $\dv$ is a descent
direction for $M(x,z_1,z_2 \given \mu)$.  Otherwise a positive-definite
modified matrix $\Hhat_k \approx H_k$ must be used.  If necessary, the
matrix $H_k$ was modified using the method of W\"{a}chter and
Biegler~\cite[Algorithm~IC, p.~36]{WacB06}, which factors the matrix $H_k +
\delta I_n$.  Each (possibly perturbed) $H_k$ matrix was factored using the
\Matlab{} built-in command \texttt{LDL}.
%, which uses the routine MA57.

\medskip
Results are presented from two variants of the Forsgren-Gill method.  The
first, \texttt{PD-Wolfe}, is the conventional primal-dual method
implemented with a Wolfe line search; the second, \texttt{PDproj-qWolfe},
is the interior-point projected-search method proposed in
Section~\ref{subsec:Interior-point-methods}.  As the underlying interior
method is the same in both cases, the results show the benefits of
formulating the method as a projected-search method.

The algorithms were considered to have solved a problem successfully if
\[
 \max \Big\{ \infnorm{ \max(0,  g(x)\dottimes x_\ell)},
             \infnorm{ \max(0,- g(x)\dottimes x_u   )}\Big\}
    \le 10^{-5},
\]
where $x_\ell = \min\left\{ 1,(x - \ell)\dotdiv (1 + \mod{\ell}) \right\}$,
$x_u = \min\left\{ 1,(u - x)\dotdiv(1 + \mod{u}) \right\}$, and $g(x) =
\Grad f(x)/(\max\{ 1, \infnorm{\Grad f(x)}\})$.  A limit of 500 was placed
on the number of iterations.  The strategy for choosing the barrier
parameter $\mu$ was that used in the method of Gertz and
Gill~\cite{GerG04}. The fraction-to-the-boundary parameter $\btol$ of
(\ref{eq:FRk-defined}) was set at $0.9$.

All testing was done using \textsc{Matlab} version R2019a on an iMac with a
3.0 GHz Intel Zeon W processor and 128 GB of 800 MHz DDR4 RAM running macOS,
version 10.14.6 (64 bit).  Results were obtained for a subset of the
bound-constrained problems in \CUTEst{} for which $n$ is of the order of 600
or less, and the dimension of the problem may be set at the largest value
less than 1000.  This gave 137 problems ranging in size from \Cute{BQP1VAR}
(one variable) to \Cute{POWELLBC} (1000 variables).  Exact second
derivatives were used for all the runs.

The full list of results is given in
Table~\ref{table:fulltestsetIP-auto}. An entry ``\texttt{--}'' indicate
that the convergence criterion was not satisfied in 500 iterations.  An
``$i$'' implies that the method was terminated because the line search was
unable to find a sufficiently improved point. (In all cases, this occurred
on iterations at which $H_k$ (\ref{eq:ip-Hessian}) was modified.)

{\texttt{\footnotesize
\begin{longtable}{|l|r|rr|rr|}
  \caption{\label{table:fulltestsetIP-auto}Interior-point methods on 137 problems.}\\\hline
  \multicolumn{2}{|c|}{}&%
  \multicolumn{2}{c|}{\texttt{PD-Wolfe}}&%
  \multicolumn{2}{c|}{\texttt{PDproj-qWolfe}\hstrt}\\\hline
  \multicolumn{1}{|c}{Problem}&%
  \multicolumn{1}{|c|}{$n$}&%
  \multicolumn{1}{c}{\texttt{Itns}}&%
  \multicolumn{1}{c|}{Nf}&%
  \multicolumn{1}{c}{\texttt{Itns}}&%
  \multicolumn{1}{c|}{Nf}\hstrt\\\hline
  \endfirsthead
  \caption[]{Comparison of interior methods on 137 problems. (Continued)} \\\hline
  \multicolumn{2}{|c|}{}&%
  \multicolumn{2}{c|}{\texttt{PD-Wolfe}}&%
  \multicolumn{2}{c|}{\texttt{PDproj-qWolfe}\hstrt}\\\hline
  \multicolumn{1}{|c}{Problem}&%
  \multicolumn{1}{|c|}{$n$}&%
  \multicolumn{1}{c}{\texttt{Itns}}&%
  \multicolumn{1}{c|}{Nf}&%
  \multicolumn{1}{c}{\texttt{Itns}}&%
  \multicolumn{1}{c|}{Nf}\hstrt\\\hline
  \endhead
  \hline
  \endfoot
%                     |   PD-Wolfe       |   PDproj-qWolfe             |
%     Problem         |    n  |   Itns   |    nF  |  Itns    |   nF    |
\Cute{3PK      }\hstrt&    30 &     10   &    11  &     6    &     7   \\
\Cute{AIRCRFTB }      &     8 &     10   &    15  &    10    &    15   \\
\Cute{ALLINIT  }      &     4 &     10   &    14  &    10    &    14   \\
\Cute{BDEXP    }      &   500 &     13   &    14  &    11    &    12   \\
\Cute{BIGGS3   }      &     6 &      8   &    14  &     8    &    14   \\
\Cute{BIGGS5   }      &     6 &     47   &    71  &    47    &    71   \\
\Cute{BIGGSB1  }      &    25 &     15   &    16  &     9    &    11   \\
\Cute{BLEACHNG }      &    17 &     20   &    22  &     5    &     6   \\
\Cute{BOX2     }      &     3 &      7   &    11  &     7    &    11   \\
\Cute{BQP1VAR  }      &     1 &     11   &    14  &    10    &    11   \\
\Cute{BQPGABIM }      &    50 &     16   &    17  &    17    &    21   \\
\Cute{BQPGASIM }      &    50 &     16   &    17  &    17    &    21   \\
\Cute{BRATU1D  }      &   503 &      4   &     5  &     4    &     5   \\
\Cute{CAMEL6   }      &     2 &      8   &    10  &     7    &     9   \\
\Cute{CHARDIS0 }      &   400 &     10   &    11  &     3    &     4   \\
\Cute{CHEBYQAD }      &   100 &     51   &    58  &    47    &    54   \\
\Cute{CLPLATEA }      &   529 &      5   &     6  &     5    &     6   \\
\Cute{CLPLATEB }      &   529 &      2   &     3  &     2    &     3   \\
\Cute{CLPLATEC }      &   529 &      1   &     2  &     1    &     2   \\
\Cute{CVXBQP1  }      &   100 &     10   &    11  &     7    &     9   \\
\Cute{DECONVB  }      &    63 &     27   &    29  &    22    &    27   \\
\Cute{DRCAV1LQ }      &   196 &     31   &    38  &    31    &    38   \\
\Cute{DRCAV2LQ }      &   196 &     20   &    30  &    20    &    30   \\
\Cute{DRCAV3LQ }      &   196 &     21   &    35  &    21    &    35   \\
\Cute{EG1      }      &     3 &     16   &    21  &    12    &    18   \\
\Cute{EXPLIN   }      &   120 &     64   &    66  &    12    &    16   \\
\Cute{EXPLIN2  }      &   120 &     68   &    70  &    12    &    16   \\
\Cute{EXPQUAD  }      &   120 &     20   &    21  &    14    &    16   \\
\Cute{FBRAIN2LS}      &     4 &     16   &    20  &    16    &    22   \\
\Cute{FBRAINLS }      &     2 &     10   &    11  &     9    &    10   \\
\Cute{GENROSEB }      &   500 &    140   &   143  &    11    &    19   \\
\Cute{GRIDGENA }      &   500 &      6   &     7  &     4    &     5   \\
\Cute{HADAMALS }      &   400 &     --   &    --  &    84    &   129   \\
\Cute{HARKERP2 }      &   500 &     29   &    30  &    14    &    15   \\
\Cute{HART6    }      &     6 &      8   &    10  &     6    &     9   \\
\Cute{HATFLDA  }      &     4 &     13   &    20  &     5    &     7   \\
\Cute{HATFLDB  }      &     4 &     11   &    12  &    12    &    13   \\
\Cute{HATFLDC  }      &    25 &      7   &     8  &     5    &     6   \\
\Cute{HIMMELP1 }      &     2 &     12   &    16  &     8    &    14   \\
\Cute{HOLMES   }      &   180 &     20   &    24  &    10    &    19   \\
\Cute{HS1      }      &     2 &     27   &    33  &    25    &    36   \\
\Cute{HS2      }      &     2 &     15   &    20  &    14    &    22   \\
\Cute{HS3      }      &     2 &     12   &    21  &     6    &     7   \\
\Cute{HS3MOD   }      &     2 &     13   &    21  &    14    &    27   \\
\Cute{HS4      }      &     2 &     15   &    23  &    10    &    13   \\
\Cute{HS5      }      &     2 &      8   &     9  &     7    &     9   \\
\Cute{HS25     }      &     3 &      0   &     1  &     0    &     1   \\
\Cute{HS38     }      &     4 &     42   &    54  &    38    &    62   \\
\Cute{HS45     }      &     5 &     11   &    12  &    10    &    12   \\
\Cute{HS110    }      &    10 &      7   &     8  &     6    &     7   \\
\Cute{JNLBRNG1 }      &   529 &     10   &    14  &     8    &    12   \\
\Cute{JNLBRNG2 }      &   529 &      9   &    13  &     7    &    11   \\
\Cute{JNLBRNGA }      &   529 &     11   &    15  &     8    &    12   \\
\Cute{JNLBRNGB }      &   529 &     12   &    16  &     7    &    11   \\
\Cute{KOEBHELB }      &     3 &     80   &   115  &    80    &   116   \\
\Cute{LINVERSE }      &   199 &     35   &    40  &    19    &    27   \\
\Cute{LMINSURF }      &   121 &      8   &    12  &     8    &    12   \\
\Cute{LOGROS   }      &     2 &     35   &    53  &    30    &    51   \\
\Cute{MAXLIKA  }      &     8 &     49   &    54  &    28    &    42   \\
\Cute{MCCORMCK }      &   500 &     11   &    12  &     6    &     7   \\
\Cute{MDHOLE   }      &     2 &     43   &    65  &    47    &    84   \\
\Cute{MINSURF  }      &    64 &      4   &     8  &     4    &     8   \\
\Cute{MINSURFO }      &   731 &     16   &    20  &    12    &    19   \\
\Cute{NCVXBQP1 }      &   100 &    439   &   442  &    14    &    17   \\
\Cute{NCVXBQP2 }      &   100 &    426   &   431  &    14    &    20   \\
\Cute{NCVXBQP3 }      &   100 &    284   &   292  &     9    &    15   \\
\Cute{NLMSURF  }      &   961 &      9   &    11  &     9    &    11   \\
\Cute{NOBNDTOR }      &   484 &     10   &    11  &     9    &    11   \\
\Cute{NONSCOMP }      &   500 &     12   &    13  &     8    &     9   \\
\Cute{OBSTCLAE }      &   529 &     14   &    17  &    11    &    14   \\
\Cute{OBSTCLAL }      &   529 &     15   &    19  &     9    &    13   \\
\Cute{OBSTCLBL }      &   529 &     62   &    63  &     8    &    10   \\
\Cute{OBSTCLBM }      &   529 &     11   &    12  &     6    &     8   \\
\Cute{OBSTCLBU }      &   529 &     82   &    83  &     8    &    10   \\
\Cute{ODC      }      &   144 &     15   &    16  &    15    &    16   \\
\Cute{OSLBQP   }      &     8 &     13   &    16  &    10    &    14   \\
\Cute{PALMER1  }      &     4 &     22   &    29  &    10    &    15   \\
\Cute{PALMER1A }      &     6 &     48   &    59  &    39    &    55   \\
\Cute{PALMER1B }      &     4 &     21   &    28  &    15    &    22   \\
\Cute{PALMER1E }      &     8 &     85   &   108  &    52    &    70   \\
\Cute{PALMER2  }      &     4 &     15   &    20  &     9    &    13   \\
\Cute{PALMER2A }      &     6 &     76   &    96  &    76    &   101   \\
\Cute{PALMER2B }      &     4 &     17   &    20  &    13    &    19   \\
\Cute{PALMER2E }      &     8 &     63   &    81  &    60    &    82   \\
\Cute{PALMER3  }      &     4 &     22   &    29  &    21    &    29   \\
\Cute{PALMER3A }      &     6 &     69   &    87  &    79    &   111   \\
\Cute{PALMER3B }      &     4 &     16   &    17  &    13    &    17   \\
\Cute{PALMER3E }      &     8 &     51   &    62  &    48    &    67   \\
\Cute{PALMER4  }      &     4 &     17   &    24  &    20    &    27   \\
\Cute{PALMER4A }      &     6 &     48   &    58  &    49    &    68   \\
\Cute{PALMER4B }      &     4 &     17   &    19  &    13    &    14   \\
\Cute{PALMER4E }      &     8 &     39   &    45  &    37    &    47   \\
\Cute{PALMER5A }      &     8 &     --   &    --  &    --    &    --   \\%itns \indef
\Cute{PALMER5B }      &     9 &    437   &   688  &   449    &   791   \\
\Cute{PALMER5D }      &     4 &      2   &     3  &     2    &     3   \\
\Cute{PALMER5E }      &     8 &     --   &    --  &    --    &    --   \\%itns singular
\Cute{PALMER6A }      &     6 &    114   &   146  &   115    &   161   \\
\Cute{PALMER6E }      &     8 &     46   &    59  &    46    &    59   \\
\Cute{PALMER7A }      &     6 &     --   &    --  &    --    &    --   \\%itns singular
\Cute{PALMER7E }      &     8 & \notdef  &\notdef &\notdef   &\notdef  \\%line search \indef
\Cute{PALMER8A }      &     6 &     35   &    44  &    33    &    43   \\
\Cute{PALMER8E }      &     8 &     28   &    30  &    23    &    29   \\
\Cute{PENTDI   }      &   500 &      9   &    14  &     7    &    12   \\
\Cute{PFIT1LS  }      &     3 &    235   &   344  &   245    &   368   \\
\Cute{PFIT2LS  }      &     3 &     36   &    50  &    33    &    54   \\
\Cute{PFIT3LS  }      &     3 &    122   &   171  &   153    &   240   \\
\Cute{PFIT4LS  }      &     3 &    222   &   325  &   260    &   397   \\
\Cute{POWELLBC }      &  1000 &     --   &    --  &   305    &    --   \\
\Cute{PROBPENL }      &   500 &      5   &     6  &     2    &     3   \\
\Cute{PSPDOC   }      &     4 &     10   &    16  &     9    &    15   \\
\Cute{QR3DLS   }      &   610 &    198   &   369  &   193    &   361   \\
\Cute{QRTQUAD  }      &   120 &     46   &    50  &    26    &    43   \\
\Cute{QUDLIN   }      &   120 &    178   &   179  &     5    &     6   \\
\Cute{RAYBENDL }      &   130 &  \notdef &\notdef &\notdef   &\notdef  \\%line search singular Hessian
\Cute{RAYBENDS }      &   130 &  \notdef &\notdef &\notdef   &\notdef  \\%line search singular Hessian
\Cute{S368     }      &   100 &     50   &    56  &    25    &    30   \\
\Cute{SANTALS  }      &    21 &     37   &    51  &    34    &    51   \\
\Cute{SCOND1LS }      &   502 &    377   &   555  &   252    &   419   \\
\Cute{SIM2BQP  }      &     2 &     12   &    13  &    11    &    13   \\
\Cute{SIMBQP   }      &     2 &      9   &    10  &     8    &    10   \\
\Cute{SINEALI  }      &   100 &     11   &    12  &     8    &     9   \\
\Cute{SPECAN   }      &     9 &     12   &    13  &     8    &    10   \\
\Cute{SSC      }      &  1122 &      2   &     3  &     2    &     3   \\% 992 free variables
\Cute{TORSION1 }      &   484 &      9   &    10  &     9    &    11   \\
\Cute{TORSION2 }      &   484 &      9   &    10  &     7    &     9   \\
\Cute{TORSION3 }      &   484 &     10   &    11  &     8    &    10   \\
\Cute{TORSION4 }      &   484 &     10   &    11  &     7    &     9   \\
\Cute{TORSION5 }      &   484 &     11   &    12  &     7    &     9   \\
\Cute{TORSION6 }      &   484 &     10   &    11  &     9    &    11   \\
\Cute{TORSIONA }      &   484 &      9   &    10  &     9    &    11   \\
\Cute{TORSIONB }      &   484 &      8   &     9  &     6    &     8   \\
\Cute{TORSIONC }      &   484 &     10   &    11  &     8    &    10   \\
\Cute{TORSIOND }      &   484 &      9   &    10  &     7    &     9   \\
\Cute{TORSIONE }      &   484 &     10   &    11  &     8    &    10   \\
\Cute{TORSIONF }      &   484 &      9   &    10  &     6    &     8   \\
\Cute{WEEDS    }      &     3 &     33   &    39  &    27    &    42   \\
\Cute{YFIT     }      &     3 &     45   &    54  &    42    &    63
\end{longtable}}}

Figure~\ref{competepp-IP} gives the performance profiles for the total
number of function evaluations required to solve the 137 problems.  The
profiles compare the primal-dual interior method \texttt{PD-Wolfe}
implemented with a Wolfe line search and a projected-search interior method
\texttt{PDproj-qWolfe} with a quasi-Wolfe line search (i.e., the method
described in Section~\ref{subsec:Interior-point-methods}).
Figure~\ref{competepp-IP} and the results of
Table~\ref{table:fulltestsetIP-auto} indicate that a projected-search
interior method with a quasi-Wolfe line search can provide substantial
improvements in robustness and performance compared to a conventional
interior method.

\begin{figure}[ht]
  \begin{center}
   \includegraphics[width=6cm,height=5.75cm]{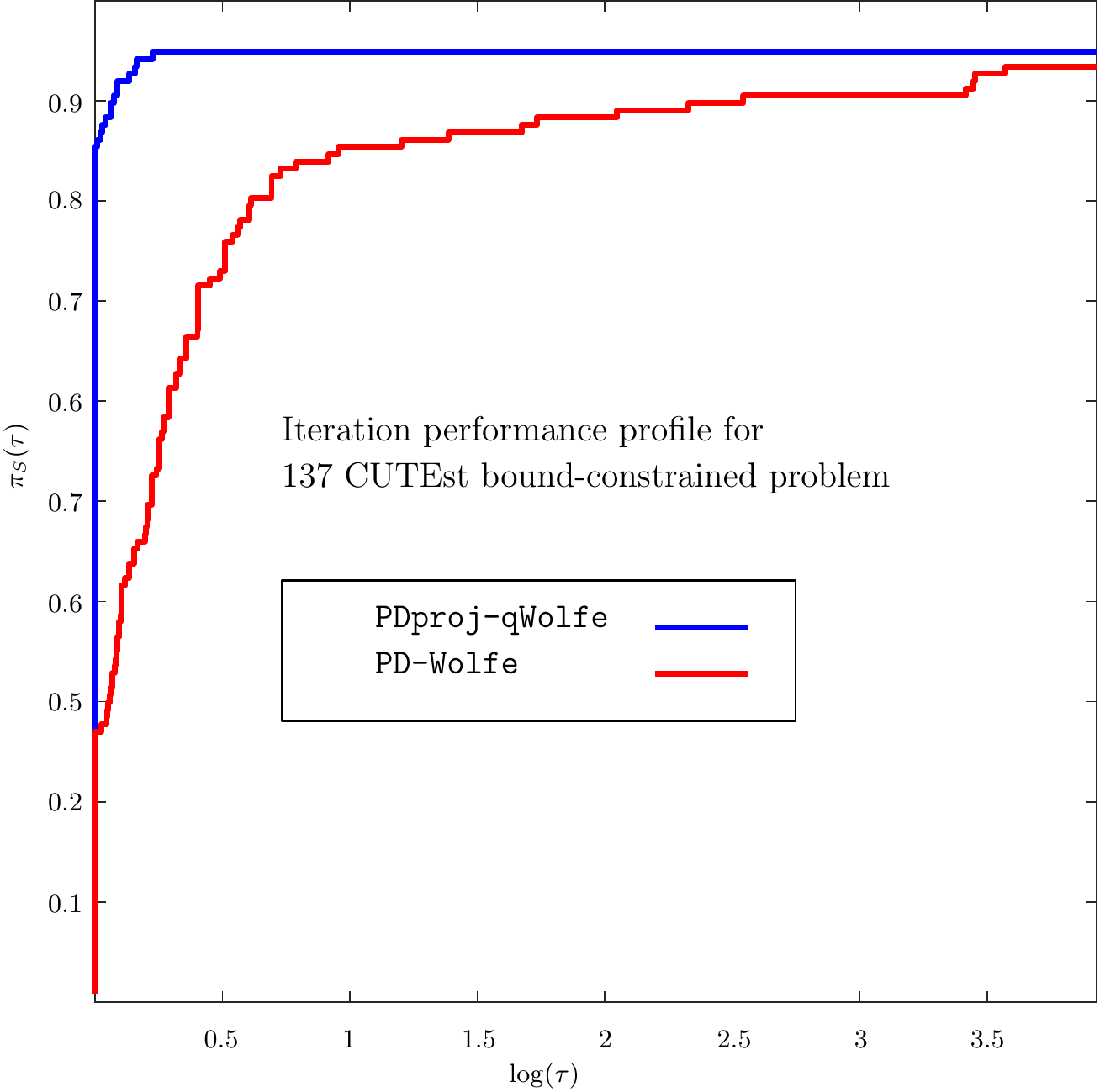}\quad
   \includegraphics[width=6cm,height=5.75cm]{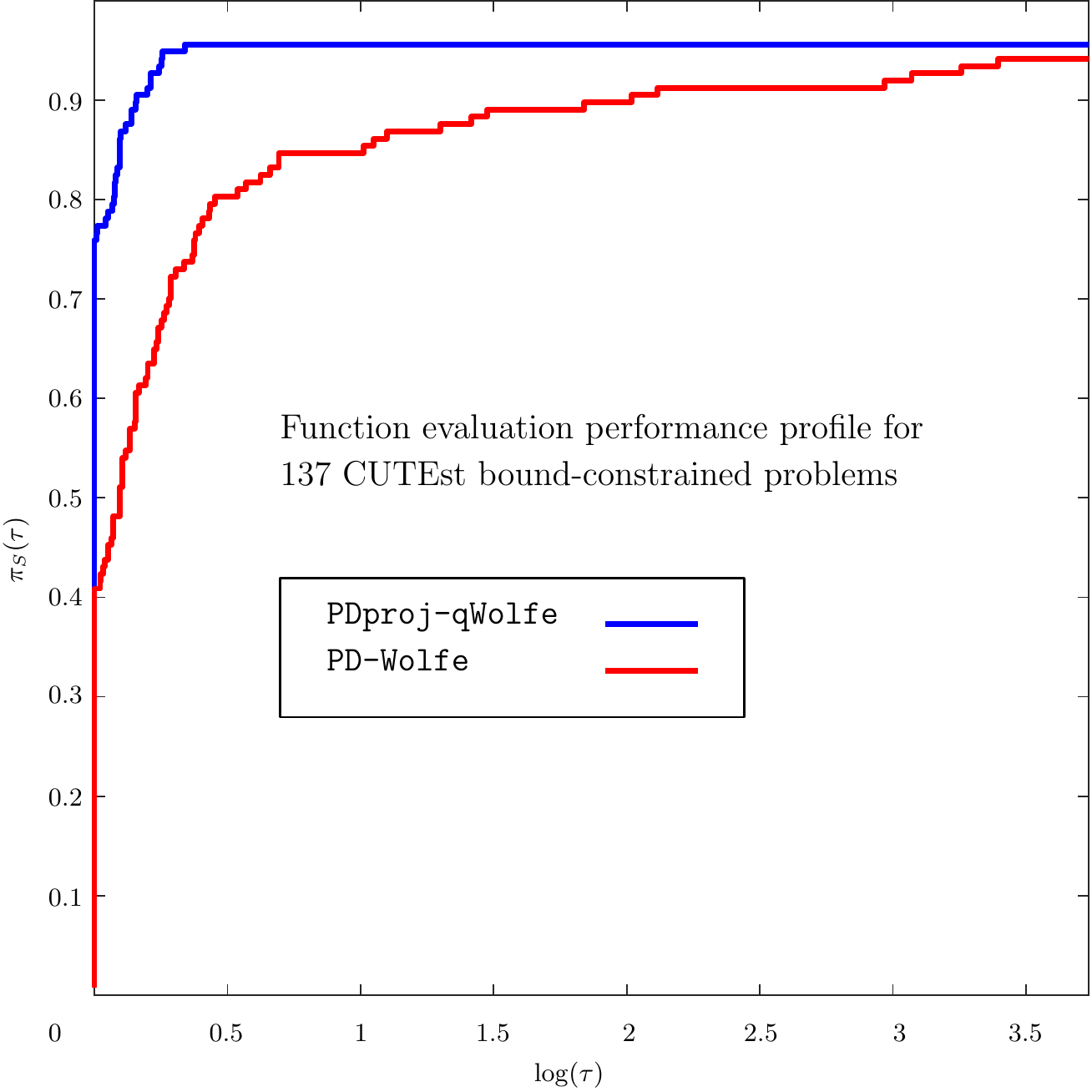}
   \caption{\label{competepp-IP}\scriptsize \small Performance profiles for
two interior-point methods: a conventional primal-dual method with a Wolfe
line search (\texttt{PD-Wolfe}) and a projected-search interior method with
a quasi-Wolfe line search (\texttt{PDproj-qWolfe}). The figure gives the
performance profiles for the total number of iterations and function
evaluations required to solve 137 bound-constrained problems from the
CUTEst test set. }
  \end{center}
\end{figure}

\section{Summary and Conclusions}\label{sec:Summary}
%--------------------------------
Two general classes of projected-search methods based on active-set methods
and interior-point methods have been proposed for bound-constrained minimization.
These methods utilize a new quasi-Wolfe search that is similar to a
conventional Wolfe line search, except that a step is accepted under a
wider range of conditions that take into consideration steps at which the
restriction of the objective function on the search path is not
differentiable. As in the unconstrained case, the quasi-Wolfe step can be
computed using safeguarded polynomial interpolation and the accuracy of the
step can be adjusted.  Standard results
associated with a conventional Wolfe line search are extended to the
quasi-Wolfe case.

The convergence of both the active-set and interior methods are established
under assumptions that are typical in the analysis of projected-search
methods.  In the case of the projected search active-set method, it is
shown that if the iterates converge to a nondegenerate stationary point,
then the optimal active set is identified in a finite number of iterations.
It follows that once the optimal active set has been identified, any method
in this class will have the same convergence rate as its unconstrained
counterpart.

Computational results are given for two specific methods: an active-set
method that uses a limited-memory quasi-Newton approximate Hessian; and an
interior-point method based on Newton's method applied to the perturbed
primal-dual optimality conditions.  The results show that in these
contexts, a quasi-Wolfe line search is substantially more efficient and
reliable than an Armijo line search.

%% Moreover, a particular method based on a limited-memory quasi-Newton method
%% to obtain the feasible descent direction is shown to be competitive with
%% the state-of-the-art package \LBFGSB\@.

\bibliographystyle{myplain}
\bibliography{references}

\def\C#1#2{C\raise3pt\hbox{\tiny#1#2}}\def\cprime{$'$} \def\cprime{$'$}
  \def\ocirc#1{\ifmmode\setbox0=\hbox{$#1$}\dimen0=\ht0 \advance\dimen0
  by1pt\rlap{\hbox to\wd0{\hss\raise\dimen0
  \hbox{\hskip.2em$\scriptscriptstyle\circ$}\hss}}#1\else {\accent"17 #1}\fi}
  \def\ocirc#1{\ifmmode\setbox0=\hbox{$#1$}\dimen0=\ht0 \advance\dimen0
  by1pt\rlap{\hbox to\wd0{\hss\raise\dimen0
  \hbox{\hskip.2em$\scriptscriptstyle\circ$}\hss}}#1\else {\accent"17 #1}\fi}
\begin{thebibliography}{10}

\bibitem{Arm66}
L.~Armijo.
\newblock Minimization of functions having {Lipschitz} continuous first partial
  derivatives.
\newblock {\em Pacific Journal of Mathematics}, 16:1--3, 1966.

\bibitem{Ber76}
D.~P. Bertsekas.
\newblock On the {G}oldstein-{L}evitin-{P}olyak gradient projection method.
\newblock {\em IEEE Trans. Automatic Control}, AC-21(2):174--184, 1976.

\bibitem{Ber82}
D.~P. Bertsekas.
\newblock {\em Constrained optimization and {L}agrange multiplier methods}.
\newblock Computer Science and Applied Mathematics. Academic Press Inc.
  [Harcourt Brace Jovanovich Publishers], New York, 1982.

\bibitem{Ber82a}
D.~P. Bertsekas.
\newblock Projected {N}ewton methods for optimization problems with simple
  constraints.
\newblock {\em SIAM J. Control Optim.}, 20(2):221--246, 1982.

\bibitem{BonCGT95}
I.~Bongartz, A.~R. Conn, N.~I.~M. Gould, and {\relax Ph}.~L. Toint.
\newblock {CUTE}: Constrained and unconstrained testing environment.
\newblock {\em ACM Trans. Math. Software}, 21(1):123--160, 1995.

\bibitem{Bre73}
R.~P. Brent.
\newblock {\em Algorithms for minimization without derivatives}.
\newblock Prentice-Hall Inc., Englewood Cliffs, N.J., 1973.
\newblock Prentice-Hall Series in Automatic Computation.

\bibitem{ByrLNZ95}
R.~H. Byrd, P.~Lu, J.~Nocedal, and C.~Zhu.
\newblock A limited memory algorithm for bound constrained optimization.
\newblock {\em SIAM J. Sci. Comput.}, 16:1190--1208, 1995.

\bibitem{CalM87}
P.~H. Calamai and J.~J. Mor{\'e}.
\newblock Projected gradient methods for linearly constrained problems.
\newblock {\em Math. Program.}, 39:93--116, 1987.

\bibitem{DenS96}
J.~E. Dennis, Jr. and R.~B. Schnabel.
\newblock {\em Numerical methods for unconstrained optimization and nonlinear
  equations}.
\newblock Society for Industrial and Applied Mathematics (SIAM), Philadelphia,
  PA, 1996.
\newblock Corrected reprint of the 1983 original.

\bibitem{DolM02}
E.~D. Dolan and J.~J. Mor{\'e}.
\newblock Benchmarking optimization software with performance profiles.
\newblock {\em Math. Program.}, 91(2, Ser. A):201--213, 2002.

\bibitem{Fer11}
M.~W. Ferry.
\newblock {\em Projected-Search Methods for Box-Constrained Optimization}.
\newblock PhD thesis, Department of Mathematics, University of California, San
  Diego, May 2011.

\bibitem{FerGWZ20b}
M.~W. Ferry, P.~E. Gill, E.~Wong, and M.~Zhang.
\newblock A limited-memory reduced-{Hessian} method for bound-constrained
  optimization.
\newblock Center for Computational Mathematics Report CCoM 20-05, Center for
  Computational Mathematics, University of California, San Diego, La Jolla, CA,
  2020.

\bibitem{FerGWZ20d}
M.~W. Ferry, P.~E. Gill, E.~Wong, and M.~Zhang.
\newblock Supplementary numerical results for projected-search methods for
  bound-constrained optimization.
\newblock Center for Computational Mathematics Report CCoM 20-08, Center for
  Computational Mathematics, University of California, San Diego, La Jolla, CA,
  2020.

\bibitem{ForG98}
A.~Forsgren and P.~E. Gill.
\newblock Primal-dual interior methods for nonconvex nonlinear programming.
\newblock {\em SIAM J. Optim.}, 8:1132--1152, 1998.

\bibitem{GerG04}
E.~M. Gertz and P.~E. Gill.
\newblock A primal-dual trust-region algorithm for nonlinear programming.
\newblock {\em Math. Program., Ser. B}, 100:49--94, 2004.

\bibitem{GhoH90}
N.~Ghosh and W.~W. Hager.
\newblock A derivative-free bracketing scheme for univariate minimization.
\newblock {\em Computers Math. Applic.}, 20(2):23--24, 1990.

\bibitem{GilMSW82}
P.~E. Gill, W.~Murray, M.~A. Saunders, and M.~H. Wright.
\newblock A note on a sufficient-decrease criterion for a nonderivative
  step-length procedure.
\newblock {\em Math. Programming}, 23(3):349--352, 1982.

\bibitem{Gol64}
A.~A. Goldstein.
\newblock Convex programming in {H}ilbert space.
\newblock {\em Bulletin of the American Mathematical Society}, 70(5):709--710,
  1964.

\bibitem{GouOT03}
N.~I.~M. Gould, D.~Orban, and {\relax Ph}.~L. Toint.
\newblock {CUTEr and SifDec}: A constrained and unconstrained testing
  environment, revisited.
\newblock {\em ACM Trans. Math. Software}, 29(4):373--394, 2003.

\bibitem{Hag89}
W.~W. Hager.
\newblock A derivative-based bracketing scheme for univariate minimization and
  the conjugate gradient method.
\newblock {\em Computers Math. Applic.}, 18(9):779--795, 1989.

\bibitem{KimSD10}
D.~Kim, S.~Sra, and I.~S. Dhillon.
\newblock Tackling box-constrained optimization via a new projected
  quasi-{N}ewton approach.
\newblock {\em SIAM J. Sci. Comput.}, 32(6):3548--3563, December 2010.

\bibitem{Knu97}
D.~Knuth.
\newblock {\em The Art of Computer Programming, 3}.
\newblock Addison-Wesley Publishing Company, Redwood City, third edition, 1997.

\bibitem{LevP66}
E.~S. Levitin and B.~T. Polyak.
\newblock Constrained minimization methods.
\newblock {\em U.S.S.R. Comput. Math. and Math. Physics}, 6(5):1--50, 1966.

\bibitem{MorN11}
J.~L. Morales and J.~Nocedal.
\newblock Remark on ``{Algorithm} 778: {L-BFGS-B}: Fortran subroutines for
  large-scale bound constrained optimization''.
\newblock {\em ACM Trans. Math. Softw.}, 38(1):7:1--7:4, December 2011.

\bibitem{MorS84}
J.~J. Mor\'{e} and D.~C. Sorensen.
\newblock Newton's method.
\newblock In G.~H. Golub, editor, {\em Studies in Mathematics, Volume 24. MAA
  Studies in Numerical Analysis}, pages 29--82. Math. Assoc. America,
  Washington, DC, 1984.

\bibitem{MorT94}
J.~J. Mor{\'e} and D.~J. Thuente.
\newblock Line search algorithms with guaranteed sufficient decrease.
\newblock {\em ACM Trans. Math. Software}, 20(3):286--307, 1994.

\bibitem{NiY97}
Q.~Ni and Y.~Yuan.
\newblock A subspace limited memory quasi-{N}ewton algorithm for large-scale
  nonlinear bound constrained optimization.
\newblock {\em Math. Comput.}, 66:1509--1520, 10 1997.

\bibitem{NocW99}
J.~Nocedal and S.~J. Wright.
\newblock {\em Numerical Optimization}.
\newblock Springer-Verlag, New York, 1999.

\bibitem{OrtR00}
J.~M. Ortega and W.~C. Rheinboldt.
\newblock {\em Iterative solution of nonlinear equations in several variables}.
\newblock Society for Industrial and Applied Mathematics (SIAM), Philadelphia,
  PA, 2000.
\newblock Reprint of the 1970 original.

\bibitem{WacB06}
A.~W{\"a}chter and L.~T. Biegler.
\newblock On the implementation of an interior-point filter line-search
  algorithm for large-scale nonlinear programming.
\newblock {\em Math. Program.}, 106(1, Ser. A):25--57, 2006.

\bibitem{Wil64}
J.~W.~J. Williams.
\newblock Algorithm 232 - {H}eapsort.
\newblock {\em Communications of the Association for Computing Machinery},
  7:347--348, 1964.

\bibitem{Wol69}
P.~Wolfe.
\newblock Convergence conditions for ascent methods.
\newblock {\em SIAM Rev.}, 11:226--235, 1969.

\bibitem{Wol72}
P.~Wolfe.
\newblock On the convergence of gradient methods under constraint.
\newblock {\em IBM J. Res. Dev.}, 16:407--411, 1972.

\bibitem{ZhuBLN97}
C.~Zhu, R.~H. Byrd, P.~Lu, and J.~Nocedal.
\newblock Algorithm 778: {L-BFGS-B}: {F}ortran subroutines for large-scale
  bound-constrained optimization.
\newblock {\em ACM Trans. Math. Software}, 23(4):550--560, 1997.

\end{thebibliography}

\end{document}